\UseRawInputEncoding
\documentclass[12pt]{amsart}
\usepackage{setspace, amsmath, amsthm, amssymb, amsfonts, amscd, epic, graphicx, ulem, dsfont}
\usepackage[T1]{fontenc}
\usepackage{multirow}
\usepackage{bbm}
\usepackage{enumerate}

\makeatletter \@namedef{subjclassname@2010}{
  \textup{2020} Mathematics Subject Classification}
\makeatother

\newtheorem{thm}{Theorem}[section]

\newtheorem{cor}[thm]{Corollary}
\newtheorem{lem}[thm]{Lemma}
\newtheorem{pro}[thm]{Proposition}

\theoremstyle{remark}
\newtheorem*{rema}{\textbf{Remark}}

\theoremstyle{definition}

\newtheorem{exa}[thm]{\textbf{Example}}

\newcommand{\im}{\operatorname{Im}}

\newcommand{\ran}{\operatorname{ran}}

\newcommand{\Real}{\operatorname{Re}}

\newcommand{\R}{\mathbb{R}}

\newcommand{\N}{\mathbb{N}}

\newcommand{\C}{\mathbb{C}}

\begin{document}

\title[Properties of $p(T)$, $TT^*$, and $T^*T$]{Certain properties involving the unbounded operators $p(T)$, $TT^*$, and $T^*T$; and some applications to powers and $nth$ roots of unbounded operators}
\author[M. H. Mortad]{Mohammed Hichem Mortad}

\date{}


\address{Laboratory of Mathematical Analysis and Applications. Department of
Mathematics, University of Oran 1, Ahmed Ben Bella, B.P. 1524, El
Menouar, Oran 31000, Algeria.}

\email{mhmortad@gmail.com, mortad.hichem@univ-oran1.dz.}

\begin{abstract}
In this paper, we are concerned with conditions under which
$[p(T)]^*=\overline{p}(T^*)$, where $p(z)$ is a one-variable complex
polynomial, and $T$ is an unbounded, densely defined, and linear
operator. Then, we deal with the validity of the identities
$\sigma(AB)=\sigma(BA)$, where $A$ and $B$ are two unbounded
operators. The equations $(TT^*)^*=TT^*$ and $(T^*T)^*=T^*T$, where
$T$ is a densely defined closable operator, are also studied. A
particular interest will be paid to the equation $T^*T=p(T)$ and its
variants. Then, we have certain results concerning $nth$ roots of
classes of normal and nonnormal (unbounded) operators. Some further
consequences and counterexamples accompany our results.
\end{abstract}

\maketitle

\section{Preliminaries}

First, we assume readers have some familiarity with linear bounded
and unbounded operators on Hilbert spaces. Some useful references
are \cite{Mortad-Oper-TH-BOOK-WSPC}, \cite{Mortad-cex-BOOK}, and
\cite{SCHMUDG-book-2012}. Classical definitions and notations follow
those in \cite{SCHMUDG-book-2012}. We do recall the definition of
the spectrum that will be used here (as this could be different
elsewhere): Let $A$ be an operator on a complex Hilbert space $H$.
The resolvent set of $A$, denoted by $\rho(A)$, is defined by
\[\rho(A)=\{\lambda\in\C:~\lambda I-A\text{ is bijective and }(\lambda I-A)^{-1}\in B(H)\}.\]
The  complement of $\rho(A)$, denoted by $\sigma(A)$, i.e.,
\[\sigma(A)=\C\setminus \rho(A)\]
is called the spectrum of $A$.

Therefore, if $A$ is a linear operator that is not closed, then
necessarily $\sigma(A)=\C$.

We also recall the definitions of unbounded nonnormal operators.

Say that a linear operator $T$ is quasinormal, provided it is
closed, densely defined, and $TT^*T=T^*TT$ (as in \cite{Kaufman
closed oeprators 1983}). This condition was weakened to
$TT^*T\subset T^*TT$ in \cite{Jablonski et al 2014}, but not
$TT^*T\supset T^*TT$. By \cite{Bernau JAusMS-1968-square root}, the
quasinormality of $T$ is, in fact, equivalent to $T|T|=|T|T$. In
\cite{Stochel-Szafraniec-normal extensions II}, $T$ is quasinormal
iff $U|T|\subset |T|U$, where $T=U|T|$ is the usual polar
decomposition of $T$ in terms of partial isometries. By, e.g., Lemma
2.2 in \cite{Ota-q deformed quasinormal JOT 2002}, it is seen that
$U|T|\subset |T|U$ is equivalent to $U|T|=|T|U$.

A densely defined linear operator $A$ with domain $D(A)\subset H$,
is said to be subnormal when there are a Hilbert space $K$ with
$H\subset K$, and a normal operator $N$ with $D(N)\subset K$ such
that
\[D(A)\subset D(N)\text{ and } Ax=Nx \text{ for all } x\in D(A).\]

A densely defined $A$ is said to be formally normal if
\[\|Ax\|=\|A^*x\|,\forall x\in D(A)\subset D(A^*).\]

A densely defined operator $A$ with domain $D(A)$ is called
hyponormal provided
\[D(A)\subset D(A^*)\text{ and } \|A^*x\|\leq\|Ax\|,~\forall x\in D(A).\]

Recall that a linear operator $A:D(A)\subset H\to H$ is said to be
paranormal when
\[\|Ax\|^2\leq \left\|A^2x\right\|\|x\|\]
for all $x\in D(A^2)$. It is straightforward to check that a
hyponormal operator is paranormal. However, and unlike the class of
hyponormal operators, a densely defined paranormal operator need not
be closable. Also, the closure of a densely defined closable
paranormal operator fails, in general, to be paranormal. See
\cite{Daniluk-paranormals-non-closable} or
\cite{Mortad-paranormal-paper-three CEX CEXEX}. In the latter
reference, an example of a densely defined paranormal operator $T$
that satisfies $D(T^*)=\{0\}$ is provided. Also, readers could find
an example of a closed densely defined operator $T$ such that both
$T$ and $T^*$ are one-to-one and paranormal, yet $T$ is not normal
(cf. \cite{Ando-Paranormal-tensor-product-sums et al} and
\cite{Yamazaki-Yanagida-PARANORMAL}).

It is widely known that a normal operator is quasinormal, a
quasinormal operator is subnormal, a subnormal operator is
hyponormal, and a hyponormal operator is paranormal. It is also
known that one cannot go backward in the last implications, even for
bounded operators (see \cite{Mortad-cex-BOOK} for counterexamples).
Readers may consult, e.g., \cite{Szafraniec Normals subnormals and
an open question} to see how all those (and more classical) notions
interplay.

\section{Introduction}

Taking the adjoint of the product or the sum of densely defined
unbounded operators can be an arduous task in some cases, and it can
be a major impediment to advances in many proofs, which makes us
frustrated before such situations. Recall that when $A$, $B$ and
$AB$ are all densely defined operators, then $A^*B^*\subset (BA)^*$,
and the equality
\[(BA)^*=A^*B^*\]
holds true if, e.g., $B\in B(H)$. It also holds when $A^{-1}\in
B(H)$.

The following lemma, which will be needed below, was shown in
Corollary 1.7 in \cite{Hardt-Mennicken-OP-Th-ADv-APP}.

\begin{lem}\label{(AB)*=B*A* (BA)*=A*B* resolvent sets LEM}
Let $A$ and $B$ be two densely defined closed operators such that
$\sigma(AB)\neq \C$ and $\sigma(BA)\neq\C$. Then $AB$ and $BA$ are
two densely defined closed operators that obey
\[(AB)^*=B^*A^*\text{ and } (BA)^*=A^*B^*.\]
\end{lem}

Some very related papers are
\cite{Azizov-Dijksma-closedness-prod-ADJ}, \cite{CG},
\cite{Gustafson-Mortad-I}, \cite{Sebestyen-Tarcsay-adj sum and
product}, and certain references therein. As for sums, readers may
consult, e.g., \cite{HK}, \cite{Mortad-CMB-2011}, and
\cite{Sebestyen-Tarcsay-adj sum and product}.

Notice that conditions like $\sigma(AB)\neq \C$ and
$\sigma(BA)\neq\C$ are strong, but they entail interesting
consequences as regards the adjoint of the product of two unbounded
operators and their spectrum. Remember that when $A,B\in B(H)$, then
$\sigma(AB)\neq \C$ and $\sigma(BA)\neq \C$. Some related papers are
\cite{azizov-denisov-philipp products self-adjoint MATH NACH},
\cite{Dehimi-Mortad-INVERT}, and \cite{Philipp-Ran-Wojtylak}.

Let $T$ be a densely defined closed operator. One of the most
fundamental properties in unbounded operator theory is the fact that
$T^*T$ (and also $TT^*$) is a densely defined, self-adjoint and
positive operator. This is a very well-known von-Neumann's theorem.
In fact, von-Neumann's result may be obtained via the so-called
Nelson's trick (cf. \cite{Thaller-Dirac-EQuation-Contains Nelson
trick}). In particular,
\[(TT^*)^*=TT^*\text{ and } (T^*T)^*=T^*T.\]
This result then enables us to define the very important notion of
the modulus of an operator which, and as it is known, intervenes in
the definition of the polar decomposition of an operator. We may
also find it in an abstract version of the damped wave equation of
the form $\ddot{u}+R\dot{u}+T^*Tu=0$ under certain conditions and in
some separable Hilbert space, where $T$ is a densely defined closed
operator and $R$ is some perturbation of $T^*T$. This equation has
been extensively investigated in \cite{Gesztesy et al. ANNALI. MATH.
AA*}. Operator theorists are also well aware of other uses of the
above result. Some related results may be found in
\cite{Boucif-Dehimi-Mortad}, , \cite{Mortad-cex-BOOK}, \cite{RS2},
and \cite{sebestyen-tarcsay-TT* has an extension}.

Recently, Z. Sebestyén and Zs. Tarcsay have discovered that if
$TT^*$ and $T^*T$ are both self-adjoint, then $T$ must be closed
(see \cite{Sebestyen-Tarcsay-TT* von Neumann T closed}, cf.
\cite{Dehimi-Mortad-squares-polynomials} and
\cite{Gesztesy-Schmudgen-AA*}). In this paper, we deal with the
validity of the identities $(TT^*)^*=TT^*$ and $(T^*T)^*=T^*T$,
where $T$ is a densely defined closable  operator. In the end, we
supply a few counterexamples that show interesting pathological
properties of $TT^*$ and $T^*T$.

But, before that, we will be first investigating the validity of the
equality $[p(T)]^*=\overline{p}(T^*)$, where $p(z)$ is a
one-variable complex polynomial, and $T$ is densely defined. Then, a
few consequences about powers of operators are obtained. Readers
might be interested in papers dealing with polynomials of closed (or
other classes of) operators such as: \cite{Dautray-Lions-VOL2},
\cite{Ota-Schmudgen-Matrix-UNBOUNDED},
\cite{SCHMUDG-1983-An-trivial-domain}, \cite{Stochel-IEOT-2002}, and
\cite{Stochel-Sza-domination-2003}.

In addition, we have some useful results concerning $nth$ roots of
certain classes of normal and nonnormal (unbounded) operators.

We close the paper with several interesting counterexamples related
to certain of the results to be shown here.

\section{On the adjoint of $p(T)$}

Before giving the general case, we present a simple consequence of
Lemma \ref{(AB)*=B*A* (BA)*=A*B* resolvent sets LEM}. First, recall
that it is well known that $(T^n)^*=T^{*n}$ for any $T\in B(H)$ and
any $n$. The unbounded version need not hold even when $n=2$. See,
e.g., Question 20.2.14 in \cite{Mortad-cex-BOOK} for a
counterexample. Remember that if $T^n$ (hence $T$) is densely
defined, then only $(T^*)^n\subset (T^n)^*$ holds. If $T$ is normal,
then one does have $(T^n)^*=T^{*n}$ for any $n$ (a proof may be
consulted in say \cite{Pietrzycki-Stochel-follow-up-2020-Conjecture
curto el al.}). Jab{\l}o\'{n}ski et al. constructed in
\cite{Jablonski et al 2014}, a quasinormal operator $T$ such that
$T^{*n}\subsetneq (T^n)^*$ for all $n\geq2$ (cf. Lemma 3.8 in
\cite{Dehimi-Mortad-squares-polynomials}).

By Lemma \ref{(AB)*=B*A* (BA)*=A*B* resolvent sets LEM}, it is seen
that when $T$ is closed and $\sigma(T^2)\neq\C$, then
$(T^2)^*=T^{*2}$. However, a naive direct generalization to the case
$T^n$, $n\in\N$, requires an induction argument, some standard
facts, and rather strong assumptions involving the closedness and
spectra of $T^p$, $p\leq n$.

In fact, we can prove a much stronger result, which will also allow
us to present a new proof of a slightly improved Lemma
\ref{(AB)*=B*A* (BA)*=A*B* resolvent sets LEM}.

Let $p(z)=a_nz^n+\cdots+a_1z+a_0$, where $a_n,\cdots, a_1,a_0$ are
complex numbers, and $z$ is a complex variable. Set
$\overline{p}(z)=\overline{a_n}z^n+\cdots+\overline{a_1}z+\overline{a_0}$.

It is easy to see that $D[p(T)]=D(T^n)$. So, when $T^n$ is densely
defined, then $T$ is a densely defined operator. But, in general,
only $[p(T)]^*\supset \overline{p}(T^*)$ holds. The coming result
seems therefore interesting.

\begin{thm}\label{ADJ of p(T) resolvent set non empty THM}
Let $n\in\N$ and let $T$ be a linear closable operator on a Hilbert
space such that $T^n$ is densely defined. If $\sigma[p(T)]\neq\C$,
then $[p(T)]^*=\overline{p}(T^*)$, where $p(z)$ is a given complex
polynomial of degree $n$.
\end{thm}

\begin{rema}
Recall that the closedness of $T$ alone does not, in general, yield
the closedness of $T^n$, neither it implies its closability. For
instance, in Question 22.2.6 in \cite{Mortad-cex-BOOK}, there is
densely defined closed operator $T$ such that $T^2$ is also densely
defined but $(T^2)^*$ is not densely defined, i.e., $T^2$ is not
closable.
\end{rema}

The proof seems to need the coming lemma, whose proof is essentially
contained in the proof of Theorem 2.1 in
\cite{Dehimi-Mortad-squares-polynomials}.

\begin{lem}\label{Deh-MOrtad-Nachr-Main THM LEMMA}
Let $p$ be a complex polynomial of one variable of degree $n$ (whose
leading coefficient equals 1, WLOG). Assume that $A$ is a closable
operator in a Banach space such that
 $\sigma[p(A)]\neq\C$. If $\lambda\in \rho[p(A)]$ is such that
\[p(A)-\lambda I=(A-\mu_1 I)(A-\mu_2I)\cdots (A-\mu_nI)\]
for certain complex numbers complex numbers $\mu_1$, $\mu_2$,
$\cdots$, $\mu_n$, then all of $A-\mu_1 I$, $A-\mu_2I$, ..., and
$A-\mu_nI$ are boundedly invertible. In particular, $\sigma(A)\neq
\C$, and so $A$ is closed.
\end{lem}

\begin{proof} By the fundamental
theorem of algebra, we know that there always exist complex numbers
$\mu_1$, $\mu_2$, $\cdots$, $\mu_n$ such that
\[p(z)-\lambda=(z-\mu_1)(z-\mu_2)\cdots (z-\mu_n),\]
where $z\in\C$. Hence
\[p(A)-\lambda I=(A-\mu_1 I)(A-\mu_2I)\cdots (A-\mu_nI)\]
with $D[p(A)-\lambda I]=D(A^n)$. In
\cite{Dehimi-Mortad-squares-polynomials}, it was shown that $A-\mu_1
I$ is boundedly invertible (hence $\sigma(A)\neq \C$, and so $A$ is
closed). But since
\[(A-\mu_1 I)(A-\mu_2I)\cdots (A-\mu_nI)=(A-\mu_2 I)(A-\mu_3I)\cdots (A-\mu_n I)(A-\mu_1I),\]
$A-\mu_2 I$ too is boundedly invertible. A similar argument gives
the bounded invertibility of the remaining terms. The proof is
complete.
\end{proof}

Now, we show Theorem \ref{ADJ of p(T) resolvent set non empty THM}:

\begin{proof}First, observe that since $T$ is closed and
$\rho(T)\neq\varnothing$, $p(T)$ is densely defined and closed (cf.
\cite{McIntosh H infity calculus D(An) dense!!}).

Let $\lambda$ be in $\C\setminus \sigma[p(T)]$. Write
\[p(z)-\lambda=(z-\mu_1)(z-\mu_2)\cdots (z-\mu_n),\]
where $\mu_1$, $\mu_2$, $\cdots$, $\mu_n$ are complex numbers. Then
\[\overline{p(z)-\lambda}=\overline{p}(\overline{z})-\overline{\lambda}=(\overline{z}-\overline{\mu_1})(\overline{z}-\overline{\mu_2})\cdots (\overline{z}-\overline{\mu_n}).\]
Symbolically,
\[p(T)-\lambda I=(T-\mu_1 I)(T-\mu_2I)\cdots (T-\mu_nI)\]
and
\[\overline{p}(T^*)-\overline{\lambda} I=(T^*-\overline{\mu_1} I)(T^*-\overline{\mu_2}I)\cdots (T^*-\overline{\mu_n}I).\]
Since each of $T-\mu_1 I$, $T-\mu_2I$, ..., and $T-\mu_nI$ is
boundedly invertible, we obtain
\begin{align*}
[p(T)]^*-\overline{\lambda} I&=(p(T)-\lambda I)^*\\
&=[(T-\mu_1 I)(T-\mu_2I)\cdots (T-\mu_nI)]^*\\
&=(T-\mu_nI)^*[(T-\mu_1 I)(T-\mu_2I)\cdots (T-\mu_{n-1}I)]^*\\
&=\cdots\\
&=(T-\mu_nI)^*\cdots (T-\mu_2I)^*(T-\mu_1I)^*\\
&=(T^*-\overline{\mu_n} I)\cdots(T^*-\overline{\mu_2}I)(T^*-\overline{\mu_1}I)\\
&=(T^*-\overline{\mu_1} I)(T^*-\overline{\mu_2}I)\cdots
(T^*-\overline{\mu_n}I)\\
&=\overline{p}(T^*)-\overline{\lambda} I.
\end{align*}
Accordingly, $[p(T)]^*=\overline{p}(T^*)$, as wished.
\end{proof}

\begin{cor}
Let $T$ be a densely defined linear closable operator on a Hilbert
space. If $\sigma[p(T)]\neq\C$, then
\[\overline{p(T)}=p(\overline{T}),\]
where $p(z)$ is a given complex polynomial.
\end{cor}

\begin{proof}By Theorem \ref{ADJ of p(T) resolvent set non empty
THM}, $[p(T)]^*=\overline{p}(T^*)$. Since $T^*$ is closed, and
$\sigma[(p(T))^*]\neq\C$, so to speak,
$\sigma[\overline{p}(T^*)]\neq\C$, it follows that
\[\overline{p(T)}=[p(T)]^{**}=[\overline{p}(T^*)]^*=\overline{\overline{p}}(\overline{T})=p(\overline{T}),\]
as desired.
\end{proof}

We have yet another consequence, which first appeared in Theorem 2.3
in \cite{Dehimi-Mortad-squares-polynomials} (cf.
\cite{Ota-Schmudgen-Matrix-UNBOUNDED},
\cite{Stochel-Sza-domination-2003}, and
\cite{Sebestyen-Tarcsay-self-adjoint squares}).

\begin{cor}\label{sigma(p(T)) real T symm T is s.a. CORO}
Let $p$ be a real polynomial of one variable of degree $n\geq 1$.
Assume that $T$ is a symmetric (not necessarily densely defined)
operator in a Hilbert space $H$ such that $p(T)$ is self-adjoint.
Then $f(T)$ is self-adjoint for any real-valued Borel function on
$\R$, where $D[f(T)]=\{x\in H:\int_{\R}|f(t)|^2d\langle
E_tx,x\rangle<\infty\}$ and $E_t$ is the associated spectral measure
with $T$.
\end{cor}

\begin{proof} Since $p$ is a non-constant polynomial, we know that there exists an $\alpha\in \C\setminus\R$ such that
$p(\alpha):=\lambda\in\C\setminus\R$ (see
\cite{Stochel-Sza-domination-2003}). By the self-adjointness of
$p(T)$, $T$ is closed. Now, let $\lambda\not\in\sigma[p(T)]$. Then
by the spectral mapping theorem (see, e.g., \cite{Gindler spec map
THM} or Section 5.7 in \cite{Taylor-Book-1958-FUNC Ana}),
$\alpha\not\in\sigma(T)$.

Then, we may write
\[p(z)-\lambda=(z-\alpha)q(z)=q(z)(z-\alpha),\]
where $q(z)$ is a polynomial of degree $n-1$. Hence
\[p(T)-\lambda I=(T-\alpha I)q(T)=q(T)(T-\alpha I)\]
and
\[p(T^*)-\lambda I=(T^*-\alpha I)q(T^*)=q(T^*)(T^*-\alpha I)\]

Since $p(T)=[p(T)]^*=p(T^*)$ and $T-\alpha I$ is boundedly
invertible, it ensues that
\[p(T)-\overline{\lambda} I=[p(T)-\lambda I]^*=(T^*-\overline{\alpha} I)[q(T)]^*.\]
The self-adjointness of $p(T)$ with a glance at the basic criterion
for self-adjointness imply
\[\ker(p(T)-\lambda I)=\ker(p(T)-\overline{\lambda} I)=\{0\},\]
from which we derive
\[\ker(T^*-\alpha I)=\ker(T^*-\overline{\alpha} I)=\{0\}.\]
Thus $T$ is self-adjoint by the basic criterion for self-adjointness
given that $T$ is already closed and symmetric. The last statement
about the self-adjoint of $f(T)$ is a mere consequence of the
spectral theorem for unbounded self-adjoint operators.
\end{proof}

The next consequence is Corollary 19 in
\cite{Stochel-Sza-domination-2003}:

\begin{cor}
Let $p$ be a real polynomial of one variable of degree $n\geq 1$.
Assume that $T$ is a symmetric operator in a Hilbert space $H$ such
that $p(T)$ is essentially self-adjoint. Then $f(\overline{T})$ is
self-adjoint for any real-valued Borel function on $\R$, where
$D[f(\overline{T})]=\{x\in H:\int_{\R}|f(t)|^2d\langle
E_tx,x\rangle<\infty\}$ and $E_t$ is the associated spectral measure
with $\overline{T}$.
\end{cor}

\begin{proof}Since a
symmetric $T$ is formally normal,  $p(\overline{T})=\overline{p(T)}$
by \cite{Stochel-Sza-domination-2003}. Hence $p(\overline{T})$ is
self-adjoint. Thus, $\overline{T}$ is self-adjoint by Corollary
\ref{sigma(p(T)) real T symm T is s.a. CORO}.

\end{proof}

\section{On the equalities $(AB)^*=B^*A^*$ and $\sigma(AB)= \sigma(BA)$ for unbounded operators}

We begin by reproving Lemma \ref{(AB)*=B*A* (BA)*=A*B* resolvent
sets LEM} under slightly weaker assumptions, namely the closability
of $A$ and $B$ instead of their closedness.

\begin{pro}\label{(AB)*=B*A* (BA)*=A*B* resolvent sets IMPROVED LEM}
Let $A$ and $B$ be two densely defined closable operators in a
Hilbert space $H$, such that $\sigma(AB)\neq \C$ and
$\sigma(BA)\neq\C$. Then $A$ and $B$ are closed. Besides, $AB$ and
$BA$ are two densely defined closed operators that obey
\[(AB)^*=B^*A^*\text{ and } (BA)^*=A^*B^*.\]
Hence
\[(A^*B^*)^*=BA\text{ and } (B^*A^*)^*=AB.\]
\end{pro}

\begin{proof}Set
\[T=\left(
      \begin{array}{cc}
        0 & A \\
        B & 0 \\
      \end{array}
    \right)
\]
with $D(T)=D(A)\oplus D(B)$. Because $A$ and $B$ are closable, $T$
is closable, too. But
\[T^2=\left(
        \begin{array}{cc}
          AB & 0 \\
          0 & BA \\
        \end{array}
      \right).
\]
By assumption, $\sigma(AB)\neq \C$ and $\sigma(BA)\neq\C$. Lemma 2.4
in \cite{Hardt-Konstantinov-Spectrum-product} then yields
$\sigma(AB)\cup\sigma(BA)\neq\C$, i.e., it is seen that
$\sigma(T^2)\neq\C$. Hence $T$ closed, from which we may derive the
closedness of $A$ as well as that of $B$. So, $T^2$ is densely
defined, which forces both $AB$ and $BA$ to be densely defined.

Finally, since $T^*=\left(
      \begin{array}{cc}
        0 & B^* \\
        A^* & 0 \\
      \end{array}
    \right)$, and as $(T^2)^*=T^{*2}$, inspection of the
entries of the two corresponding operator matrices gives the desired
identities
\[(AB)^*=B^*A^*\text{ and } (BA)^*=A^*B^*,\]
as needed.

To obtain the last claim, just take adjoints, by remembering that
$AB$ and $BA$ are closed.
\end{proof}

The next similar result is stated without proof.

\begin{cor}
Let $T$ and $S$ be linear closable operators on a Hilbert space. Let
$p$ and $q$ be two one-variable complex polynomials of degrees $n$
and $m$, respectively. If $\sigma[p(T)]\neq\C$ and
$\sigma[q(S)]\neq\C$, $\sigma[p(T)q(S)]\neq\C$, and
$\sigma[q(S)p(T)]\neq\C$, then $p(T)q(S)$ and $q(S)p(T)$ are densely
defined, and
\[[p(T)q(S)]^*=\overline{q}(S^*)\overline{p}(T^*)\text{ and }[q(S)p(T)]^*=\overline{p}(T^*)\overline{q}(S^*).\]
\end{cor}

It is well-known that an expression like $AB=BA$, on some common
dense domain, does not always yield the strong commutativity of $A$
and $B$ (which are assumed to be unbounded and self-adjoint). This
is the famous Nelson's counterexample (see
\cite{Nelson-Analytic-vectors} or \cite{RS1}). Nelson's example is
developed in detail in \cite{Schmudgen-Operator-algebra}, pp
257-258. The simplest Nelson-like example in the literature is due
to K. Schm\"{u}dgen in \cite{Schmudgen-Nelson LIKE}. So, the next
consequence of Proposition \ref{(AB)*=B*A* (BA)*=A*B* resolvent sets
IMPROVED LEM} is worth stating.

\begin{cor}(Cf. \cite{DevNussbaum}, \cite{Mortad-OaM-2014})
Let $A$ and $B$ be two unbounded self-adjoint operators. If $AB=BA$
and $\sigma(AB)\neq\C$, then $A$ strongly commutes with $B$.
\end{cor}

\begin{proof}By assumption, $\sigma(AB)\neq\C$ and
$\sigma(BA)\neq\C$. Hence $(AB)^*=BA=AB$, i.e., $AB$ is
self-adjoint. By, e.g. \cite{DevNussbaum}, $A$ and $B$ strongly
commute.
\end{proof}

It is widely known that
\[\sigma(AB)-\{0\}=\sigma(BA)-\{0\}\]
for any $A,B\in B(H)$, with $\dim H=\infty$ (cf. \cite{Deift} and
\cite{Hardt-Konstantinov-Spectrum-product}). Imposing certain
conditions on $A$ and/or $B$ gives the full equality
$\sigma(AB)=\sigma(BA)$. For instance, when $A$ is normal. See
\cite{Barraa-Boumazghour}. However, such a result was at least known
at the time of \cite{Cho-Curto-Huruya- sigma(Ab)=sigma(BA) A
normal}, and this was missed by some authors, including myself. See
\cite{Hladnik-Omladic-spectrum-product-PAMS-1988} and
\cite{Mortad-Oper-TH-BOOK-WSPC} for other instances of when this is
true, and \cite{Mortad-cex-BOOK} for related counterexamples. It
appears, however, that the most general condition guaranteeing the
equality of spectra is $\ker A=\ker A^*$ (or $\ker B=\ker B^*$).
This appeared in \cite{Dehimi-Mortad-INVERT}, in the case $B$ is in
$B(H)$ and $A$ is a densely defined, non-necessarily bounded, and
closed operator. Below we generalize this result to two unbounded
operators.

\begin{thm}\label{sigma(AB)=sigma(BA) A B TWO UNBD THM}
Let $A$ and $B$ be two densely defined closable operators in a
Hilbert space $H$, such that $\sigma(AB)\neq \C$ and
$\sigma(BA)\neq\C$. If $\ker (A^*)=\ker(A)$ and $\ker(B)=\ker(B^*)$,
then
\[\sigma(BA)=\sigma(AB).\]
\end{thm}

\begin{rema}
The preceding theorem improves Equation 3.2 in Proposition 3.1 in
\cite{azizov-denisov-philipp products self-adjoint MATH NACH}.
\end{rema}

\begin{proof}By Proposition \ref{(AB)*=B*A* (BA)*=A*B* resolvent sets IMPROVED LEM}, $A$ and $B$ are closed. Then, Theorem 1.1 in
\cite{Hardt-Konstantinov-Spectrum-product} gives
$\sigma(AB)-\{0\}=\sigma(BA)-\{0\}$. What remains to show is:
\[BA \text{ is boundedly invertible} \Longleftrightarrow AB \text{ is boundedly invertible}.\]
Assume that $BA$ is boundedly invertible, and so $BA$ is right
invertible. That is, $BAT=I$ for some $T\in B(H)$. Since $AT$ is
closed and $D(AT)=H$, $AT\in B(H)$. So, $B$ is right invertible, but
because $\ker B\subset \ker B^*$, it follows that $B$ is boundedly
invertible (Theorem 2.3 in \cite{Dehimi-Mortad-INVERT}). Since
$(BA)^*$ is boundedly invertible, so is $A^*B^*=(BA)^*$. In
particular, $A^*B^*$ is right invertible, and a similar argument as
above yields the right invertibility of $A^*$. Since $\ker
A^*\subset \ker A$, Theorem 2.3 in \cite{Dehimi-Mortad-INVERT}
implies the bounded invertibility of $A^*$, or that of $A$. In other
words, $AB$ is boundedly invertible.

Conversely, suppose $AB$ is boundedly invertible. Hence
$(AB)^*=B^*A^*$ too is boundedly invertible. Now, apply the first
part of the proof together with the conditions $\ker B^*\subset \ker
B$ and $\ker A\subset \ker A^*$ to obtain that $A^*B^*$ is boundedly
invertible. Accordingly, $BA=(A^*B^*)^*$ is boundedly invertible,
establishing the result.
\end{proof}

\begin{cor}\label{spectrum AB =BA UNBD normal A B CORO}Let $A$ and $B$ be two unbounded normal operators such that $\sigma(AB)\neq \C$ and
$\sigma(BA)\neq\C$. Then
\[\sigma(BA)=\sigma(AB).\]
\end{cor}

The following simple and related observation could interest some
readers:

\begin{pro}Let $A$ and $B$ be two self-adjoint operators such that
$B\in B(H)$ is positive and $BA$ is closed. Then
\[\sigma(AB)\neq\C\Longleftrightarrow \sigma(AB)\subset\R.\]
\end{pro}

\begin{proof}
We are only concerned with the implication "$\Rightarrow$", so
suppose $\sigma(AB)\neq\C$. Let $\sqrt B$ be the unique positive
square root of $B$. Since $BA$ is closed on $D(A)$, by Proposition
3.7 in \cite{Dehimi-Mortad-INVERT}, $\sqrt BA$ is also closed. Since
$\sqrt BA$ is domain-dense, $(\sqrt BA)^*=A\sqrt B$. By the
closedness of $\sqrt BA$, $A\sqrt{B}$ is densely defined. Therefore,
\[\sqrt BA=(\sqrt BA)^{**}=(A\sqrt{B})^*.\]
Since $\sqrt B\in B(H)$, it is seen that
\[(\sqrt BA\sqrt B)^*=(A\sqrt B)^*\sqrt B=\sqrt B A \sqrt B,\]
which means that $\sqrt B A \sqrt B$ is self-adjoint. So
$\sigma(\sqrt B A \sqrt B)\subset\R$ or $\sigma(\sqrt B A \sqrt
B)\neq\C$. Since $\sigma(AB)\neq\C$ means that $\sigma(A\sqrt B\sqrt
B)\neq\C$, we infer that
\[\sigma(AB)-\{0\}=\sigma(A\sqrt B\sqrt
B)-\{0\}=\sigma(\sqrt B A \sqrt B)-\{0\}\] (utilizing Theorem 1.1 in
\cite{Hardt-Konstantinov-Spectrum-product}). Thus,
$\sigma(AB)\subset\R$, as needed.
\end{proof}

\begin{rema}
It is worth recalling that if the closedness of $BA$ is replaced by
the normality of $AB$ in the result above (resp. the normality of
$BA$), then $AB$ is self-adjoint (resp. $BA$ and $AB$ are
self-adjoint). In particular, $\sigma(AB)\subset\R$. See \cite{MHM1}
for a proof. Related results about unbounded operators may be
consulted in \cite{Dehimi-Mortad-Bachir-Comm-Closed-Symm-OPER},
\cite{Gustafson-Mortad-II}, \cite{Jung-Mortad-Stochel}, and
\cite{MHM7}.
\end{rema}

Similarly, we have:

\begin{pro}
Let $A$ and $B$ be two self-adjoint operators such that $B\in B(H)$
is positive. If $\sqrt B$ has a closed range, and has finite
dimensional kernel, then
\[\sigma(AB)\neq\C\Longleftrightarrow \sigma(AB)\subset\R.\]
\end{pro}

\begin{proof}By \cite{Gustafson-BAMS-PAP self-adjoint}, $\sqrt B A \sqrt B$ is
self-adjoint. The rest of the proof is as above.
\end{proof}

\section{On $nth$ roots of some classes of unbounded operators}

It is plain that if $T\in B(H)$, then $T^2$ is self-adjoint if and
only if $T^2={T^*}^2$. This is not the case anymore when $T$ is
closed and densely defined. Indeed, there is a closed densely
defined operator $T$ such that
\[D(T^2)=D({T^*}^2)=\{0\}\]
(hence $T^2$ cannot be self-adjoint). Such an example may be found
in \cite{Dehimi-Mortad-CHERNOFF}. The situation is not better when
$D(T^2)$ is dense (see \cite{Dehimi-Mortad-squares-polynomials}),
neither it is for higher powers (witness \cite{Mortad-TRIVIALITY
POWERS DOMAINS}). Nonetheless, we have:

\begin{cor}\label{kerha 31/07/2022 COROLLARY}
Let $n\in\N$ be fixed and let $T$ be a linear closable operator on a
Hilbert space. Then
\[T^n\text{ is self-adjoint }\Longleftrightarrow T^n=T^{*n} \text{ and }\sigma(T^n)\neq\C.\]
\end{cor}

\begin{proof}If $T^n$ is self-adjoint, then $\sigma(T^n)\neq\C$ and
\[T^n=(T^n)^*=T^{*n}.\]
Conversely, $T^n=T^{*n}$ and $\sigma(T^n)\neq\C$ imply
$(T^n)^*=T^{n}$. Consequently, $T^n$ is self-adjoint.
\end{proof}

Before giving a result that generalizes Proposition 3.7 in
\cite{Dehimi-Mortad-squares-polynomials}, recall the ensuing result,
which might be known under different proofs.

\begin{lem}\label{Quasinormal all powers QUASI LEMMA}
Let $T$ be a quasinormal (unbounded) operator. Then $T^n$ is closed
and densely defined, for each $n\in\N$.
\end{lem}

\begin{proof}Let $n\in\N$. Write $T=U|T|=|T|U$, where $U\in B(H)$ is a partial isometry. Then
\[T^n=(U|T|)^n=U^n|T|^n=|T|^nU^n.\]
Since $T$ is closed, $|T|$ is self-adjoint, as are its powers
$|T|^n$. Since $D(T^n)=D(|T|^n)$, $T^n$ is densely defined.

Since $|T|^nU^n$ is closed for each $n$, $T^n$ too is closed.
\end{proof}

Here is the promised generalization.

\begin{pro}\label{GHGHjjjklllm pro 000.}
Let $T$ be a quasinormal (unbounded) operator such that
$T^n={T^*}^n$ for some natural number $n$. Then $T^{n}$ is
self-adjoint, and $T$ is normal.

In particular, and when $n=2$, then $T$ is either self-adjoint or
skew-adjoint.
\end{pro}

\begin{proof}

We claim that
\[|T|^{2n}=|T^*|^{2n}=T^{2n}.\]

By Corollary 3.1 in \cite{Uchiyama-1993-QUASINORMAL}, we know that
\[T^{*n}T^n=(T^*T)^n\geq (TT^*)^n\geq T^nT^{*n},\]
whenever $T$ is quasinormal, and for all $n\geq2$. Hence
\[T^{2n}=(T^*T)^n=|T|^{2n}\geq |T^*|^{2n}=(TT^*)^n\geq T^{2n},\]
still for any $n\geq2$. Therefore, we obtain
$|T|^{2n}=|T^*|^{2n}=T^{2n}$.

So, $T^{2n}$ is self-adjoint and positive. Since $T$ is quasinormal,
$T^n$ is quasinormal and densely defined. Since a quasinormal
operator is hyponormal, Corollary 3.6 in
\cite{Dehimi-Mortad-squares-polynomials} gives the self-adjointness
of $T^n$.

Now, since $|T|^{2n}=|T^*|^{2n}$, upon passing to the unique
self-adjoint positive $n$-th root, we get $|T|^2=|T^*|^2$, that is,
$T$ is normal.

Finally, let $n=2$. Then $T^2$ is self-adjoint and $T$ is normal.
Let $\lambda\in\sigma(T)$, and so $\lambda^2\in\sigma(T^2)\subset
\R$. Then, either $\im \lambda=0$ or $\Real \lambda=0$. In the
latter case, $T$ is skew-adjoint, and in the former case, $T$ is
self-adjoint. This marks the end of the proof.
\end{proof}

\begin{rema}
Another way of seeing that $T^n$ is self-adjoint is: The condition
$T^n={T^*}^n$ implies that $T^n$ is symmetric. Since $T^n$ is
quasinormal, the corollary to Theorem 3.3 in
\cite{Uchiyama-1993-QUASINORMAL} gives the self-adjointness of
$T^n$.
\end{rema}

From the previous proof and Theorem \ref{ADJ of p(T) resolvent set
non empty THM}, we have:

\begin{cor}
Let $n\in\N$ and let $T$ be a quasinormal (unbounded) operator. Then
\[T^n\text{ is self-adjoint }\Longleftrightarrow T^n=T^{*n}.\]
\end{cor}

The next result to be given below these lines is already known in
the case of normal operators, and it is most probably known to
specialists in the case of paranormal operators, which is the
version to be shown. We include a simple proof for ease of
reference.

\begin{lem}\label{ker T=ker Tn T paranormal UNBD LEMM}
Let $n\in\N$ and let $T$ be a non-necessarily bounded paranormal
operator. Then
\[\ker T=\ker T^2=\cdots=\ker T^n.\]
Hence $\overline{\ran T^n}=\cdots=\overline{\ran T^2}=\overline{\ran
T}$.
\end{lem}

\begin{proof}First, observe that we always have $\ker T\subset\ker
T^n$ for any $n$ (and for \textit{any} linear operator). Indeed, let
$x\in \ker T$, i.e., $x\in D(T)$ and $Tx=0$. Hence $Tx\in
D(T^{n-1})$ and $T^nx=0$, i.e., $x\in \ker T^n$.

To show the converse, recall that from \cite{Szafraniec paranormal
powers kernels}, we know that for all $x\in D(T^{n+1})$, the
following inequalities hold:
\[\|T^nx\|\leq \|T^{n+1}x\|^{n/(n+1)}\|x\|^{1/(n+1)}\]
and
\[\|Tx\|\leq \|T^{n+1}x\|^{1/(n+1)}\|x\|^{n/(n+1)}.\]
Let $x\in\ker T^n$, that is, $T^nx=0$ with $x\in D(T^n)$ (hence
$x\in D(T)$). Therefore, $x\in\ker T$, as wished.

The last statement holds by standard arguments.
\end{proof}

\begin{cor}
Let $T$ be a quasinormal (unbounded) operator such that
$T^n={T^*}^n$ for some natural $n$. Then $T$ is normal.
\end{cor}

\begin{proof}Since $T$ is quasinormal and $T^n={T^*}^n$, it is seen
that
\[\ker T=\ker T^n=\ker T^{*n}\supset \ker T^*.\]
But, a quasinormal operator is hyponormal, and so $\ker T\subset
\ker T^*$, i.e. $\ker T=\ker T^*$. By Corollary 3.2 in
\cite{Uchiyama-1993-QUASINORMAL}, $T$ is normal.

\end{proof}

The next consequence appeared in Theorem 3.9 in
\cite{Dehimi-Mortad-squares-polynomials}. The proof here is,
however, much shorter.

\begin{cor}\label{Tn nrmal T quasinormal yields T normal CORO}
Let $T$ be a quasinormal (unbounded) operator such that $T^n$ is
normal for a certain $n\geq 2$. Then $T$ is normal.
\end{cor}

\begin{proof}
Since $T$ is quasinormal, $\ker T=\ker T^n$. Since $T^n$ is normal,
$\ker (T^n)=\ker [(T^n)^*]$. As $T^{*n}\subset (T^n)^*$, $\ker
T^*\subset \ker [(T^n)^*]$. Therefore, $\ker T^*\subset \ker T$, and
because we already have $\ker \subset \ker T^*$, we obtain $\ker
T^*=\ker T$. Thus, and by invoking Corollary 3.2 in
\cite{Uchiyama-1993-QUASINORMAL}, $T$ is normal, as needed.
\end{proof}

\begin{rema} There is yet another way of showing the last
result. According to Corollary 3.3 in
\cite{Uchiyama-1993-QUASINORMAL}, we obtain the normality of $T$
from the quasinormality of $T$, $\ker T^*=\ker T^{*2}$, and $\ker
T^*\subset D(T^*T)$. Let us then show the last two conditions by
supposing $T$ is quasinormal and $T^n$ is normal for some $n\geq2$:

That $\ker T^*\subset \ker T^{*2}$ is clear. Then
\[\ker T^{*2}\subset \ker T^{*n}\subset \ker (T^n)^*=\ker T^n=\ker T\subset \ker T^*,\]
whereby $\ker T^*=\ker T^{*2}$.

Now, let $x\in \ker T^*$. As in the above proof, $x\in\ker T$, i.e.,
$x\in D(T)$ and $Tx=0$. Hence $Tx\in D(T^*)$, that is, $x\in
D(T^*T)$, as needed.
\end{rema}

Corollary \ref{Tn nrmal T quasinormal yields T normal CORO} has a
certain generalization to polynomials.

\begin{pro}\label{Quasin T p(T) normal gives T normal PRO}Let $T$ be a quasinormal (unbounded) operator in a complex Hilbert space $H$. Let $p(z)$ be a non-constant complex polynomial such that $p(0)=0$, and which is expressed as
\[p(z)=z^{n_1}(z-\lambda_2)^{n_2}\cdots(z-\lambda_k)^{n_k},\]
where $0,\lambda_2,\cdots,\lambda_k$ are distinct roots. Assume that
$\lambda_2,\dots,\lambda_k\not\in\sigma_p(T)$. If $p(T)$ is normal,
then $T$ is normal.
\end{pro}

\begin{proof}We have
\[p(T)=T^{n_1}(T-\lambda_2I)^{n_2}\cdots(T-\lambda_kI)^{n_k}.\]
By \cite{Gonzalez ker p(T)} (or Lemma 3.1.15 in
\cite{Laursen-Neumann ker p(T)}), it follows that
\[\ker p(T)=\ker(T^{n_1})\oplus \ker(T-\lambda_2I)^{n_2}\oplus\cdots\oplus\ker(T-\lambda_kI)^{n_k}.\]
Since a quasinormal operator is hyponormal, it ensues that all of
$T$, $T-\lambda_2I, \cdots, T-\lambda_kI$ are hyponormal. So, Lemma
\ref{ker T=ker Tn T paranormal UNBD LEMM} implies
\[\ker p(T)=\ker T\oplus \ker(T-\lambda_2I)\oplus\cdots\oplus\ker(T-\lambda_kI).\]
But $\lambda_2,\dots,\lambda_k\not\in\sigma_p(T)$, and so
\[\ker(T-\lambda_2I)=\cdots=\ker(T-\lambda_kI)=\{0_H\}.\]
Therefore, $\ker p(T)=\ker T$. By the normality of $p(T)$, we know
that $\ker p(T)=\ker[(p(T))^*]$. Since $p(0)=0$, it is seen that
\[\ker T^*\subset \ker \overline{p}(T^*)\subset \ker[(p(T))^*]=\ker p(T)=\ker T.\]
Consequently, $T$ is normal by proceeding as in the end of the proof
of Corollary \ref{Tn nrmal T quasinormal yields T normal CORO}.
\end{proof}

The verbatim extension to symmetric operators is also valid. We
need, however, to add an assumption but then the conclusion is
stronger.

First, we have an auxiliary lemma, which characterizes
self-adjointness. It is interesting enough to be singled out.

\begin{lem}\label{T sym clo ker T=ker T* ran T clo T s.a. LEM}
Let $T$ be an unbounded, densely defined, and symmetric operator in
a Hilbert space $H$, i.e., $T\subset T^*$. If $\ran T$ is closed and
$\ker T=\ker T^*$, then $T$ is self-adjoint.
\end{lem}

\begin{proof}Observe that
\[\ker T=\ker T^*=(\ran T)^{\perp}.\]
Hence $(\ker T)^{\perp}=\ran T$ thanks to the closedness of $\ran
T$. Therefore, $\ker T\oplus \ran T=H$. By Corollary 6.5 in
\cite{Sebestyen-Tarcsay-LMA-2019}, $T$ is self-adjoint.
\end{proof}

\begin{rema}There is another way to establish the preceding lemma.
First, remember that $\ran T$ is closed if and only if $\ran T^*$ is
closed (see Theorem 5.13, IV, \cite{Kat}). Then apply Proposition
2.1 in \cite{SEBES-TARCS-CHARATER-SELF6ADJ-OPER}.
\end{rema}

\begin{pro}\label{T symm p(T) normal T s.a. PRO}
Let $T$ be an unbounded, densely defined, and symmetric operator in
a complex Hilbert space $H$ such that its range $\ran T$ is closed.
Let $p(z)$ be a non-constant complex polynomial such that $p(0)=0$,
and which is expressed as
\[p(z)=z^{n_1}(z-\lambda_2)^{n_2}\cdots(z-\lambda_k)^{n_k},\]
where $0,\lambda_2,\cdots,\lambda_k$ are distinct roots. Assume that
$\lambda_2,\dots,\lambda_k\not\in\sigma_p(T)$. If $p(T)$ is normal,
then $T$ is self-adjoint.
\end{pro}

\begin{proof}Since a symmetric operator is hyponormal, obtain, as in
the proof of Proposition \ref{Quasin T p(T) normal gives T normal
PRO}, that $\ker T=\ker T^*$. Then, apply Lemma \ref{T sym clo ker
T=ker T* ran T clo T s.a. LEM} to get the self-adjointness of $T$,
as wished.
\end{proof}

The following result gives a simple characterization of orthogonal
projections (cf. \cite{SEBES-TARCS-CHARATER-SELF6ADJ-OPER}).

\begin{pro}Let $T$ be a symmetric operator with domain $D(T)\subset H$, a priori non-densely
defined. Assume $T^2\subset T$. Then the following assertions are
equivalent:
\begin{enumerate}
  \item $T\in B(H)$ is an orthogonal projection,
  \item $T^n$ is self-adjoint for some $n\geq2$.
\end{enumerate}
\end{pro}

\begin{proof} We only show the implication "$(2)\Rightarrow (1)$".
Since $T^2\subset T$, it follows that $T^n\subset T^2\subset T$ for
any $n\geq2$. Since $T^n$ is self-adjoint, $T$ is densely defined.
So, $T$ is a densely defined symmetric operator. By maximality,
$T^n=T$, and so $T$ is self-adjoint. Besides, $T=T^2$, i.e., $T$ is
idempotent. Since $D(T)=D(T^2)$ and $T$ is self-adjoint, it ensues
that $T\in B(H)$ (by Lemma 2.1 in \cite{Sebestyen-Stochel-JMAA},
say), as desired.

\end{proof}

It is known that the condition $D(A)=D(B)$ does not imply
$D(A^2)=D(B^2)$, even when $A$ and $B$ are self-adjoint positive
operators. See Question 21.2.51 in \cite{Mortad-cex-BOOK} for a
counterexample. However, it is shown in (\cite{Weidmann}, Theorem
9.4) that if $A$ and $B$ are two self-adjoint positive operators
with domains $D(A)$ and $D(B)$ respectively, then
  \[D(A)=D(B)\Longrightarrow D(\sqrt{A})=D(\sqrt{B}).\]
It is therefore natural to wonder whether this property remains
valid for arbitrary square roots. That is, if $A$ and $B$ are square
roots of some $S$, i.e. $A^2=B^2=S$, is it true that $D(A)=D(B)$?
This question, which is meaningless over $B(H)$, has a negative
answer for unbounded operators. See Example \ref{square roots unbd
s.a. diffe. domain EXA}.

Nonetheless, if the square roots are self-adjoint then they
necessarily have equal domains. In fact, a much better result holds
true, and for arbitrary $nth$ roots. It reads:

\begin{pro}
Let $A$ and $B$ be two quasinormal operators such that $A^n=B^n$ for
some $n\in\N$, $n\geq2$. Then $D(A)=D(B)$.
\end{pro}

\begin{proof}Since $A$ and $B$ are quasinormal, we have
\[A^n=B^n\Longrightarrow |A|^n=|A^n|=|B^n|=|B|^n\]
(see, e.g., Corollary 3.8 in \cite{Jablonski et al 2014}). Upon
passing to the unique positive self-adjoint $nth$ root implies
$|A|=|B|$. Therefore, $D(A)=D(B)$, as desired.
\end{proof}

The generalization to unbounded closed hyponormal operators seems
difficult to obtain (if it is ever true), however, we have the
following result, whose proof is omitted.

\begin{pro}\label{20/07/2020}
Let $A$ and $B$ be two (closed) hyponormal operators such that
$A^2=B^2$. Assume further that $A^2$ is self-adjoint and positive.
Then $D(A)=D(B)$.
\end{pro}

We finish this section with a result in the same spirit but the
approach is different. More precisely, we show that a square root of
a symmetric operator is normal, if some conditions are imposed on
its real and imaginary parts.

\begin{thm}\label{main THMMMMMMMMMMMMMMMM SQ RT}(Cf.
\cite{Frid-Mortad-Dehimi-nilpotence}) Let $T=A+iB$ where $A$ and $B$
are self-adjoint (one of them is also positive). If $T^2$ is
symmetric, then
\[T \text{ is normal }\Longleftrightarrow D(A)=D(B)\text{ and }D(AB)=D(BA).\]

If this is the case, then $T^2$ is self-adjoint. If we further
assume the positiveness of $T^2$, then $T$ is self-adjoint and
positive.
\end{thm}

Before proving this result, we give some auxiliary result whose
proof is very simple, thus omitted. It is worth noticing in passing
that there are unbounded self-adjoint operators $A$ and $B$ such
that $A+iB\subset 0$ (where 0 designates the zero operator on all of
$H$), yet $A\not\subset 0$ and $B\not\subset 0$. For example, let
$A$ and $B$ be unbounded self-adjoint operators such that $D(A)\cap
D(B)=\{0_H\}$ (as in, e.g., \cite{Mortad-cex-BOOK}). Assuming
$D(A)=D(B)$ makes the whole difference. Indeed:

\begin{pro}\label{kkkkkkkkkkkkkkkkkkkkkkkkkkk}
Let $A$ and $B$ be two densely defined symmetric operators with
domains $D(A),D(B)\subset H$ respectively. Assume that $D(A)=D(B)$.
If $A+iB\subset 0$, then $A\subset 0$ and $B\subset 0$. If $A$ (or
$B$) is further taken to be closed, then $A=B=0$ everywhere on $H$.
\end{pro}

Now, we prove Theorem \ref{main THMMMMMMMMMMMMMMMM SQ RT}.

\begin{proof}
The implication "$\Rightarrow$" follows from one of the versions of
the spectral theorem for unbounded normal operators.

Now, assume $D(A)=D(B)$ and $D(AB)=D(BA)$. Also, suppose $A$ is
positive (the proof in the case of the positiveness of $B$ is
similar).

We have
\[A^2-B^2+i(AB+BA)\subset (A+iB)A+i(A+iB)B=T^2,~~  \footnote{~In fact "$\subset$" is a full equality due to the conditions on domains, but this observation does not help much for the rest of the proof.}\]
thereby
\[A^2-B^2-T^2+i(AB+BA)\subset 0.\]
Since $D(A)=D(B)$, $D(A^2)=D(BA)$ and $D(B^2)=D(AB)$. Hence
\[D(A^2-B^2)=D(AB+BA).\]
Since $D(AB)=D(BA)$, we have
\[D(A^2-B^2-T^2)=D(AB+BA)=D(A^2)=D(B^2).\]

Since $A$ is self-adjoint, $A^2$ is densely defined. Then both
$AB+BA$ and $A^2-B^2-T^2$ are densely defined. By the symmetricity
 of $A$,  $B$, and $T^2$, we have
\[AB+BA\subset (AB+BA)^*\text{ and }A^2-B^2-T^2\subset (A^2-B^2-T^2)^*.\]
That is, both $AB+BA$ and $A^2-B^2-T^2$ are symmetric. Proposition
\ref{kkkkkkkkkkkkkkkkkkkkkkkkkkk} then yields $AB+BA\subset 0$.
Hence $AB=-BA$ (for $D(AB)=D(BA)$) and so
\[A^2B=-ABA=BA^2.\]
As $A$ is positive, we obtain $AB=BA$ by \cite{Bernau
JAusMS-1968-square root}. Hence $(A+I)B\subset B(A+I)$. But
$D[B(A+I)]=\{x\in D(A):Ax+x\in D(B)\}$. So, if $x\in D[B(A+I)]$, it
follows that $x\in D(A)=D(B)$ and $Ax\in D(B)$, i.e., $x\in D(BA)$.
Since $D(AB)=D(BA)$, we have $x\in D(AB)=D[(A+I)B]$. Thus
\[(A+I)B=B(A+I).\]

Since $A$ is self-adjoint and positive, it results that $A+I$ is
boundedly invertible. Then $(A+I)^{-1}B\subset B(A+I)^{-1}$. By
Proposition 5.27 in \cite{SCHMUDG-book-2012}, this means that $A$
commutes strongly with $B$. Accordingly $T$ is normal.

Hence, $T^2$ is normal, and so it is self-adjoint because it is
already symmetric.

Finally, we show the last statement. Assume that $T^2$ is positive,
and remember that $T$ is normal and $T^2$ is self-adjoint. Let
$\lambda$ be a complex number in $\sigma(T)$. Then
\[\lambda^2\in
[\sigma(T)]^2=\sigma(T^2).\] That is, $\lambda^2\geq0$ and so the
only possible outcome is $\lambda\in\R$. Therefore, $T$ is
self-adjoint. Since in this case
\[0\leq A=\Real T=\frac{T+T^*}{2}=T,\]
it follows that $T$ is also positive. This marks the end of the
proof.
\end{proof}

\begin{cor}(Cf. \cite{Putnam-sq-rt-normal})
Let $T=A+iB$ where $A$ and $B$ are self-adjoint (one of them is also
positive) where $D(A)=D(B)$. If $T^2=0$ on $D(T)$, then $T\in B(H)$
is normal and so $T=0$ everywhere on $H$.
\end{cor}

\begin{proof}
What prevents us a priori from using Theorem \ref{main
THMMMMMMMMMMMMMMMM SQ RT} is that the condition $D(AB)=D(BA)$ is
missing. But, writing $A=(T+T^*)/2$ and $B=(T-T^*)/{2i}$ (and so
$D(T)\subset D(T^*)$), we see that if $x\in D(T)$, then
\[Tx+T^*x\in D(T)\Longleftrightarrow Tx-T^*x\in D(T)\]
for $Tx\in D(T)$ (because $D(T^2)=D(T)$). In other language,
$D(AB)=D(BA)$, as needed.
\end{proof}

\section{On the operator equation $T^*T=p(T)$}

In \cite{Dehimi-Mortad-Tarcsay-1}, it was shown that if $A$ is
densely defined closed operator that satisfies the equation
$A^2=A^*A$, then $A$ is self-adjoint. This result now is a
consequence of Theorem \ref{ADJ of p(T) resolvent set non empty
THM}. Recall that such a problem was first mooted in
\cite{Laberteux-A*A=A2} (see also \cite{Wang-Zhang} and
\cite{McCullough-Rodman-A*A=A2}, cf.
\cite{Roman-Sandovici-Mortad-Dehimi-Tarcsay} and \cite{Tian DMT
AA*=A2}).

\begin{cor}\label{T*T=T2 T closed COROLL}
Let $T$ be a closed and densely defined operator verifying
$T^*T=T^2$. Then $T$ is self-adjoint.
\end{cor}

\begin{proof}Clearly
\[T^*T=T^2\Longrightarrow TT^*T=T^{3}\Longrightarrow TT^*T=T^2T\Longrightarrow TT^*T=T^*TT,\]
from which we derive the quasinormality of $T$. Since $T^2$ is
self-adjoint, Theorem \ref{ADJ of p(T) resolvent set non empty THM}
yields $T^2=T^{*2}$. By Proposition \ref{GHGHjjjklllm pro 000.}, $T$
is normal. The way of showing that $\sigma(T)\subset\R$ is as in the
proof of Theorem 3.2 in \cite{Dehimi-Mortad-Tarcsay-1}. Thus, $T$ is
self-adjoint.
\end{proof}

\begin{cor}\label{TT*=T2 T closed COROLL}
Let $T$ be a densely defined closed operator satisfying
$TT^*=T^{2}$. Then $T$ is self-adjoint.
\end{cor}

\begin{proof}Because $T$ is closed, it is seen that $TT^*$, or $T^2$, is
self-adjoint. Theorem \ref{ADJ of p(T) resolvent set non empty THM},
used in the special case $p(z)=z^2$, then gives
$T^2=(T^2)^*=T^{*2}$. Hence $TT^*=T^{*2}$. Setting $S=T^*$, which is
closed and densely defined, the previous equation becomes
$S^*S=S^2$. Therefore, $S$ must be self-adjoint by Corollary
\ref{T*T=T2 T closed COROLL}. Thus,
\[T=\overline{T}=T^{**}=T^*,\]
as wished.
\end{proof}

\begin{cor}\label{T*T=T*2 T closed COROLL}
Let $T$ be a densely defined closed operator satisfying
$T^*T=T^{*2}$. Then $T$ is self-adjoint.
\end{cor}

\begin{cor}\label{TT*=T*2 T closed COROLL}
Let $T$ be a densely defined closed operator satisfying
$TT^*=T^{*2}$. Then $T$ is self-adjoint.
\end{cor}

We have an additional result about square roots of symmetric
operators.

\begin{cor}\label{symm SQ RT T*2 sym IFF T ess s.a. COR}Let $T$ be a densely defined symmetric operator. Then
\[T \text{ is essentially self-adjoint} \Longleftrightarrow T^{*2} \text{ is symmetric}.\]
\end{cor}

\begin{proof}Assume $T^{*2}$ is symmetric. Since $T$ is symmetric,
it ensues that $TT^*\subset \overline{T}T^*\subset T^{*2}$. So,
$T^{*2}$ is densely defined as $\overline{T}T^*$ is self-adjoint. By
the same token, $\overline{T}T^*\subset T^{*2}$ becomes
$\overline{T}T^*=T^{*2}$ for self-adjoint operators are maximally
symmetric.  The preceding equation may be rewritten as
$\overline{T}T^*=(\overline{T})^{*2}$. Therefore, $\overline{T}$ is
self-adjoint by Corollary \ref{TT*=T*2 T closed COROLL}.

The converse, which is plain, is left to readers.
\end{proof}

\begin{rema}
It is worth noticing that the equation $T^*T=T^2$ yields, thanks to
Theorem \ref{ADJ of p(T) resolvent set non empty THM},
$T^*T=T^{*2}$. Hence $(T^*T)^2=T^{*2}T^2$ ($=T^2T^{*2}$). This
equation has been investigated in detail in \cite{Jablonski et al
2014}.
\end{rema}

What about the more general equation $T^*T=p(T)$? We will deal
directly with the case of densely defined closed operators.

\begin{thm}\label{T*T=p(T) THM}Let $T$ be a densely defined closed operator with domain
$D(T)\subset H$, and such that
\[T^*T=p(T),\]
where $p(z)$ is a polynomial of degree $n\geq1$. Then the following
assertions hold:
\begin{enumerate}
  \item $T$ is quasinormal.
  \item If $n\neq 2$, then $T\in B(H)$.
  \item\label{3 00000000.} Let $n=2$. Assume that $p(z)$ has the property that whenever
  there is (a priori complex) $z$ such that $p(z)\geq0$, then $z\in\R$. Then $T$ is self-adjoint.
\end{enumerate}
\end{thm}

\begin{proof}\hfill
\begin{enumerate}
  \item The observation
  \[T^*T=p(T)\Longrightarrow TT^*T=Tp(T)=p(T)T=T^*TT\]
shows the quasinormality of $T$.
  \item If $n=1$, then $D(T^*T)=D(T)$, which gives $T\in B(H)$
  because $T$ is closed (Exercise 11 on Page 427 in
  \cite{Dieudonne-ANALysis-II}, see also Lemma 2.1 in
  \cite{Sebestyen-Stochel-JMAA}).

  Assume now that $n\geq 3$. Since $T$ is quasinormal and closed, $|T^n|=|T|^n$. Then
  \[D(|T|^2)=D(T^*T)=D[p(T)]=D(T^n)=D(|T^n|)=D(|T|^n),\]
  and because $|T|$ is self-adjoint and positive, Lemma A.1 in
  \cite{Sebestyen-Stochel-JMAA} yields $|T|\in B(H)$. Therefore, for
  a certain $M\geq0$ and all $x\in H$
  \[\|Tx\|=\||T|x\|\leq M\|x\|,\]
  that is, $T\in B(H)$.
  \item First, recall that a closed hyponormal operator with real
  spectrum is self-adjoint (see the proof of Theorem 8 in
  \cite{Dehimi-Mortad-BKMS}). It is also known that quasinormality
  yields hyponormality. So, we need only show that
  $\sigma(T)\subset\R$. Let $\lambda\in\sigma(T)$. Then,
  \[p(\lambda)\in p[\sigma(T)]=\sigma[p(T)]=\sigma(T^*T)\subset
  [0,\infty).\]
  By the hypothesis on $p(z)$, we immediately obtain
  $\sigma(T)\subset \R$, as desired.
\end{enumerate}
\end{proof}

\begin{rema}There are certain interpretations of the condition on
$p(z)$ above. For instance, it could be replaced by
$p(z)=az^2+bz+c$, with $a>0$ and $4ac-b^2\leq0$, as kindly indicated
by Prof. Spiros Konstantogiannis. In the counterexamples section, I
give simple examples showing the indispensability of these
conditions on the coefficients.
\end{rema}

\begin{exa}
Let $T$ be a densely defined closed operator such that
$T^*T=T^2-3T+2I$. Then $T$ is self-adjoint. We need only check that
$\sigma(T)\subset \R$. Let $\lambda:=\alpha+i\beta\in\sigma(T)$.
Then $\lambda^2-3\lambda+2\geq0$. For the inequality to make sense,
we must have $\beta=0$ (while $\alpha^2-3\alpha+2\geq0$). In
consequence, $\lambda\in\R$, as needed.
\end{exa}

\begin{rema}
It is clear that $T^*T=p(T)$, where $\deg p(z)=1$, directly gives
the self-adjointness of $T$ when $p$ has real coefficients only. If
$p(z)$ has at least one complex coefficient, then $T$ is normal.
\end{rema}

\begin{cor}
Let $T$ be a densely defined, closed, and symmetric operator. If
\[T^*T=p(T),\]
where $p(z)$ is a polynomial of degree $n\geq1$, then $T$ is
self-adjoint.
\end{cor}

\begin{proof} By Theorem \ref{T*T=p(T) THM}, $T$ is quasinormal, and
we have already recalled that a quasinormal symmetric operator is
self-adjoint.
\end{proof}

\begin{cor}
Let $T$ be a densely defined, closed, and symmetric operator. If
\[TT^*=p(T),\]
where $p(z)$ is a polynomial of degree $n\geq1$, then $T$ is
self-adjoint.
\end{cor}

\begin{proof} In this case, we cannot get the quasinormality of $T$
out of $TT^*=p(T)$. Nonetheless, $T$ is self-adjoint by Corollary
\ref{sigma(p(T)) real T symm T is s.a. CORO}.
\end{proof}

\section{On the adjoint of $TT^*$ and $T^*T$}

Let $T$ be densely defined and closable. We will deal with the
validity of, e.g., $(TT^*)^*=TT^*$, where $TT^*$ is densely defined.
Recall that $TT^*$ is symmetric. If $TT^*$ is further assumed to be
closed, then, as is known, either $\sigma(TT^*)=\C$; or
$\sigma(TT^*)=\{\lambda\in\C:\im \lambda\geq0\}$; or
$\sigma(TT^*)=\{\lambda\in\C:\im \lambda\leq0\}$; or
$\sigma(TT^*)\subset\R$. By assuming below that $\sigma(TT^*)\neq
\C$, and so $TT^*$ is closed; we show that the only option left is
$\sigma(TT^*)\subset\R$.

\begin{thm}\label{ADJ AA* A unclosed THM}Let $T$ be a densely
defined closable operator. If $TT^*$ is densely defined, then
\[\sigma(TT^*)\neq \C\Longleftrightarrow (TT^*)^*=\overline{T}T^*=TT^*.\]

Similarly, when $T^*T$ is densely defined, then
\[\sigma(T^*T)\neq \C\Longleftrightarrow(T^*T)^*=T^*\overline{T}=T^*T.\]

\end{thm}

\begin{proof}Clearly
\[\overline{T}T^*\subset (TT^*)^*.\]
Since $\sigma(TT^*)\neq \C$, consider a complex number $\lambda$
such that $TT^*-\lambda I$ is (boundedly) invertible. Then
\[(TT^*-\lambda I)^*=(TT^*)^*-\overline{\lambda} I\]
remains invertible.

On the other hand, since $\overline{T}T^*$ is self-adjoint,
$\overline{T}T^*-\overline{\lambda} I$ too is (boundedly)
invertible. Therefore
\[\overline{T}T^*\subset (TT^*)^*\Longrightarrow \overline{T}T^*-\overline{\lambda} I\subset (TT^*)^*-\overline{\lambda} I.\]
Since, e.g., $\overline{T}T^*-\overline{\lambda} I$ is surjective
and $(TT^*)^*-\overline{\lambda} I$ is injective, Lemma 1.3 in
\cite{SCHMUDG-book-2012} gives $\overline{T}T^*-\overline{\lambda}
I=(TT^*)^*-\overline{\lambda} I$ or merely $\overline{T}T^*=
(TT^*)^*$.

To show the other equality, we may reason as above by using
$TT^*\subset \overline{T}T^*$. Alternatively, here is a different
approach: Since $\overline{T}T^*$ is self-adjoint, we have
\[\overline{T}T^*=(\overline{T}T^*)^*=(TT^*)^{**}=\overline{TT^*}=TT^*\]
where the closedness of $TT^*$ is obtained from $\sigma(TT^*)\neq
\C$. This settles the implication "$\Rightarrow$".

Conversely, if $(TT^*)^*=\overline{T}T^*=TT^*$, then because $TT^*$
is self-adjoint and positive, $\sigma(TT^*)\subset [0,\infty)$, that
is, $\sigma(TT^*)\neq \C$.

The second assertion of the theorem may be shown analogously, thus
omitted.
\end{proof}

\begin{rema}
Recall that a condition like $\sigma(A)\neq \C$ does not necessarily
give the density of $D(A)$ (e.g., let $A:=I_E$ be the identity
operator restricted to an unclosed non-dense domain $E$).
\end{rema}

\begin{rema}
Despite the fact the previous proof is short and elementary, there
is yet another proof of the implication "$\Rightarrow$", which was
communicated to me by a reader, who prefers to remain anonymous. It
reads: By assumptions, $S:=TT^*$ is a densely defined, and symmetric
operator with non-empty resolvent set, so $S$ is closed. Since $S$
is positive as well, it has equal deficiency indices which, together
with $\sigma(S)\neq \C$, imply $\sigma(S)\subset \R$. Hence $S$ is
self-adjoint.
\end{rema}

The following consequence is presented without proof.

\begin{cor}Let $T$ be a densely
defined closable operator. If $TT^*$ is densely defined, then
\[\sigma(TT^*)\neq \C\Longleftrightarrow \sigma(TT^*)\subset [0,\infty).\]

Similarly, when $T^*T$ is densely defined, then
\[\sigma(T^*T)\neq \C\Longleftrightarrow\sigma(T^*T)\subset [0,\infty).\]
\end{cor}

By scrutinizing the previous proof, it turns out that we may show a
stronger result, given that we have not required the closedness of
the spectrum in this paper.

\begin{pro}
Let $T$ be a densely defined closable operator. If $TT^*$ is densely
defined and $(TT^*)^*$ does not possess at least one complex
eigenvalue, then
\[(TT^*)^*=\overline{T}T^*=TT^*.\]
In particular, $TT^*$ is self-adjoint.
\end{pro}

\begin{rema}
Mutatis mutandis, a similar result holds for $T^*T$.
\end{rema}

\begin{proof}It suffices to let $\lambda$ be a complex number such
that $(TT^*)^*-\overline{\lambda} I$ is injective, and the rest of
the proof is as that of Theorem \ref{ADJ AA* A unclosed THM}.
\end{proof}

\begin{rema}
When $\sigma(TT^*)\neq\C$ and $\sigma(T^*T)\neq \C$ simultaneously,
then both $TT^*$ and $T^*T$ are densely defined, and besides
\[(TT^*)^*=TT^*\text{ and } (T^*T)^*=T^*T.\]
This may be consulted in \cite{Hardt-Konstantinov-Spectrum-product}
and \cite{Hardt-Mennicken-OP-Th-ADv-APP}.
\end{rema}

A. Devinatz, A. E. Nussbaum, and J. von Neumann obtained in
\cite{DevNussbaum-von-Neumann} the following maximality result (cf.
\cite{Meziane-Mortad-I}):

\begin{thm}\label{Devinatz-Nussbaum-von Neumann: T=T1T2}
Let $T$, $T_1$ and $T_2$ be self-adjoint operators. Then
\[T\subset T_1T_2 \Longrightarrow T=T_1T_2.\]
\end{thm}

A somewhat similar idea as above may be reused to reach an akin
conclusion, namely:

\begin{pro}
Let $T$, $T_1$ and $T_2$ be unbounded operators. If $T$ is
self-adjoint and $T\subset T_1T_2$, then
\[\sigma(T_1T_2)\neq\C \Longleftrightarrow T=T_1T_2.\]
\end{pro}

\begin{proof}
Pick a complex $\lambda$ such that $T_1T_2-\lambda I$ is boundedly
invertible. By the self-adjointness of $T$, $T-\lambda I$ is always
boundedly invertible. So
\[T\subset T_1T_2\Longrightarrow T-\lambda I\subset T_1T_2-\lambda I.\]
Hence $(T-\lambda I)^{-1}=(T_1T_2-\lambda I)^{-1}$, as they are both
everywhere defined. Thus, $T-\lambda I=T_1T_2-\lambda I$, or merely
$T=T_1T_2$. This shows the implication "$\Rightarrow$". The backward
implication is evident, and the proof is complete.
\end{proof}

\section{Counterexamples}

\begin{enumerate}
  \item (Cf. Theorem \ref{T*T=p(T) THM}) The equation $T^*T=p(T)$, if $\deg p(z)=2$ (and without any other condition on $p(z)$), does not
always yield the self-adjointness of $T$, even when $\dim H<\infty$.
Remember that in a such context, quasinormality coincide with
normality, so the matrix must first be normal.
\begin{exa}
On a bi-dimensional space, consider
\[T=\left(
      \begin{array}{cc}
        0 & 1 \\
        -1 & 0 \\
      \end{array}
    \right).
\]
If $p(z)=z^2+2$, then it is seen that
\[T^*T=\left(
         \begin{array}{cc}
           1 & 0 \\
           0 & 1 \\
         \end{array}
       \right)=T^2+2I.
\]
Observe in the end that the complex number $i$, which is an
eigenvalue for $T$, is such that $p(i)\geq0$.
\end{exa}

The next example shows that even if $p(z)$ has two real roots, we
still require the positivity of the leading coefficient. Indeed:
\begin{exa}
Let $T$ be a non-zero skew-adjoint matrix, i.e. $T^*=-T$ and
$T^*\neq T$. Then $T^*T=-T^2$, yet $T$ is not self-adjoint.
\end{exa}

\item Corollary 3.2 in \cite{Uchiyama-1993-QUASINORMAL} was
extremely helpful to establish some of the results in this paper.
Recall that it states that a quasinormal operator $T$ is normal if
(and only if) $\ker T=\ker T^*$.

It would therefore be practical to have a result like it for other
classes of operators. The answer is, alas, negative. We first give
an example borrowed from Answer 21.2.5 in \cite{Mortad-cex-BOOK},
where readers can find more details about it.

\begin{exa}Consider $Tf(x)=xf(x)$ on
\[D(T)=\left\{f\in L^2(\R):xf\in
L^2(\R),\int_{\R}f(x)dx=0\right\}.\] Then $T$ is densely defined,
closed, and symmetric. However, $T$ is not self-adjoint, i.e.,
$D(T)\neq D(T^*)$. It is also seen that
\[\ker T=\ker T^*=\{0\}.\]
\end{exa}
In other words, the condition $\ker T=\ker T^*~(=\{0\})$ combined
with the symmetricity and the closedness of $T$ do not suffice to
force $T$ to be self-adjoint. In fact, $T$ is not even normal.

Now, a densely defined symmetric operator is formally normal. Also,
a symmetric operator always has a self-adjoint extension, possibly
in a larger Hilbert space (see, e.g., Proposition 3.17 in
\cite{SCHMUDG-book-2012}). Put differently, a symmetric operator is
subnormal, and it is patently hyponormal. That being said, and due
to the above example, a formally normal or subnormal operator $T$
with the condition $\ker T=\ker T^*~(=\{0\})$ need not be normal.
\item As alluded to above, below we give an unbounded self-adjoint
positive operator whose square roots have pairwise different
domains.

\begin{exa}\label{square roots unbd s.a. diffe. domain EXA}
Let $T$ be an unbounded, self-adjoint, and positive operator with
domain $D(T)\subsetneq H$, then define $S=\left(
                    \begin{array}{cc}
                      T & 0 \\
                      0 & T \\
                    \end{array}
                  \right)$, where $D(S)=D(T)\oplus D(T)$.
It is plain that $S$ too is unbounded, self-adjoint, and positive.
Letting $\sqrt{T}$ represent the unique self-adjoint positive square
root of $T$, readers may readily check that each of
\[A=\left(
                        \begin{array}{cc}
                          0 & T \\
                          I & 0 \\
                        \end{array}
                      \right),B=\left(
                        \begin{array}{cc}
                          0 & I \\
                          T & 0 \\
                        \end{array}
                      \right) \text{, and }C=\left(
                    \begin{array}{cc}
                      \sqrt{T} & 0 \\
                      0 & \sqrt{T} \\
                    \end{array}
                  \right)\]
is in effect a square root of $S$. Observe in the end that
$D(A)=H\oplus D(T)$, $D(B)=D(T)\oplus H$, and
$D(C)=D(\sqrt{T})\oplus D(\sqrt{T})$. In other words, these three
domains are pairwise different.
\end{exa}

  \item It is natural to ask whether the identity $(TT^*)^*=TT^*$, which
holds for densely defined closed operators $T$, still holds for
closable operators.  The following example answers this question in
the negative.
\begin{exa}\label{Tarcsay'sssssssssssssssssss EXA}

Let $A$ be a densely defined positive operator in a Hilbert space
$H$ such that it is not essentially self-adjoint. Assume further
that ran$A$ is dense in $H$. Define an inner product space on $\ran
A$ by
\[\langle Ax,Ay\rangle _{A}=\langle Ax,y\rangle,~x,y\in D(A).\]

Denote the completion of this pre-Hilbert space by $H_A$. Define the
canonical embedding operator $T:H_A\supseteq \text{ran}A\to H$ by
\[T(Ax):=Ax,~x\in D(A).\]
It may then be shown that $D(A)\subset D(T^*)$ and
\[T^*x=Ax\in H_A,~x\in D(A).\]
In particular, $T:H_A\supset \ran A\to H$ is a densely defined and
closable linear operator. Besides, $T^{**}T^*$ is a self-adjoint
positive extension of $A$.

On the other hand, notice that
\[\ker T^*=(\ran T)^{\perp}=(\ran A)^{\perp}=\{0\}.\]
In other words, $T^*$ is one-to-one. So, $T^*y\in\ran A$ (=$D(T)$),
where $y\in D(T^*)$, implies $y\in D(A)$. Accordingly,
\[D(TT^*)=\{y\in D(T^*):T^*y\in D(T)\}=D(A).\]
Thus, $TT^*=A$. This signifies that
\[(TT^*)=A^*\neq T^{**}T^*=\overline{T}T^*\]
for the latter operator is self-adjoint, whilst the former is not.
\end{exa}
  \item The preceding example may be beefed up to obtain a stronger
counterexample, namely:

\begin{exa}
Let $T$ be a closable densely defined operator such that
$(TT^*)^*\neq \overline{T}T^*$ (as just before), then set
\[S=\left(
      \begin{array}{cc}
        0 & T \\
        T^* & 0 \\
      \end{array}
    \right)
\]
with $D(S)=D(T^*)\oplus D(T)$. So $S$ is densely defined, and since
$S^*=\left(
      \begin{array}{cc}
        0 & \overline{T} \\
        T^* & 0 \\
      \end{array}
    \right)$, $S$ is symmetric.

Now,
\[SS^*=\left(
         \begin{array}{cc}
           TT^* & 0 \\
           0 & T^*\overline{T} \\
         \end{array}
       \right)\text{ and }\overline{S}S^*=\left(
         \begin{array}{cc}
           \overline{T}T^* & 0 \\
           0 & T^*\overline{T} \\
         \end{array}
       \right).\]
But
\[(SS^*)^*=\left(
         \begin{array}{cc}
           (TT^*)^* & 0 \\
           0 & (T^*\overline{T})^* \\
         \end{array}
       \right)=\left(
         \begin{array}{cc}
           (TT^*)^* & 0 \\
           0 & T^*\overline{T} \\
         \end{array}
       \right)
\]
because $T^*\overline{T}$ is self-adjoint. Since we already know
that $(TT^*)^*\neq \overline{T}T^*$, it ensues that $(SS^*)^*\neq
\overline{S}S^*$ and yet $S$ is densely defined and symmetric.
\end{exa}
  \item We have been assuming that $TT^*$ is densely defined as, in general,
$D(TT^*)$ could be non-dense. For instance, in
\cite{Mortad-TRIVIALITY POWERS DOMAINS} (or \cite{Mortad-cex-BOOK}),
we have found an example of a densely defined $T$ that obeys
\[D(T^2)=D(T^*)=D(TT^*)=D(T^*T)=\{0\}.\]

Obviously such an operator $T$ cannot be closable. So, does the
closability of $T$ suffice to make $TT^*$ densely defined? The
answer is negative even when $T$ is symmetric. This is seen next:

\begin{exa}
Let $E$ be a dense linear proper subspace of $H$, let $u$ be a
non-zero element in $H$ but not in $E$, and define $T$ to be
projection on the 1-dimensional subspace spanned by $u$ with
$D(T)=E$. Then $T$ is a bounded, non-everywhere defined, unclosed,
and symmetric operator. Also, $T^*$ is the same projection, defined
on the entire $H$. Then $T^*x$ is in $E$ only if it is 0. So,
\[D(TT^*)=\{u\}^{\perp},\] which is not dense.
\end{exa}

\begin{rema}
In the previous example, one sees that $D(T^*T)$ is dense. This is,
however, not peculiar to this example. Indeed, if $T$ is a densely
defined, symmetric and bounded (non-everywhere defined) operator
with domain $D(T)\subsetneq H$, then $D(T^*T)$ is always dense. This
is easy to see, as $T\subset T_H$ where $T_H$ is the extension of
$T$ to all of $H$. Then $T_H^*\subset T^*$, and so $D(T^*)=H$.
Therefore,
\[D(T^*T)=D(T),\]
which is dense.
\end{rema}
\item Last but not least, we supply a densely defined (unclosed) symmetric
operator $T$ such that $T^*T$ is not densely defined. The same idea
could have been used to construct an akin counterexample in the case
of $TT^*$.

\begin{exa}
Let $A$ and $B$ be self-adjoint operators, where $B\in B[L^2(\R)]$,
yet to be chosen. Then, define
\[T=\left(
      \begin{array}{cc}
        0 & AB \\
        BA & 0 \\
      \end{array}
    \right),
\]
with $D(T)=D(A)\oplus D(AB)$. A priori, $BA$ need not be closed, and
we will choose $BA$ as such. We also take $AB$ to be densely
defined. Then $T$ will be densely defined, and symmetric for
\[T^*=\left(
      \begin{array}{cc}
        0 & AB \\
        (AB)^* & 0 \\
      \end{array}
    \right)\supset \left(
      \begin{array}{cc}
        0 & AB \\
        BA & 0 \\
      \end{array}
    \right)=T.\]
Moreover,
\[T^*T=\left(
         \begin{array}{cc}
           AB^2A & 0 \\
           0 & (AB)^*AB \\
         \end{array}
       \right).
\]
We are done as soon as $D(AB^2A)$ is not dense. Now, we give a pair
of $A$ and $B$ that meet the above conditions. Some details about
the remaining portion of this example may have to be consulted in
Answer 21.2.24 in \cite{Mortad-cex-BOOK} (and also \cite{KOS}).

Consider the self-adjoint operator $S$
\[Sf(x)=e^{\frac{x^2}{2}}f(x),\]
defined on $D(S)=\{f\in L^2(\R):~e^{\frac{x^2}{2}}f\in L^2(\R)\}$.
Then, set $A:=\mathcal{F}^*S\mathcal{F}$, where $\mathcal{F}$ is the
usual $L^2(\R)$-Fourier transform. So, $A$ is also self-adjoint.

Clearly $S$ is boundedly invertible (hence so is $A$). Set
\[B^2f(x):=S^{-1}f(x)=e^{\frac{-x^2}{2}}f(x),\]
which is defined from $L^2(\R)$ onto $D(S)$.

We know that $D(AB^2)$ is trivial if $D(A)\cap \ran(B^2)=\{0\}$, and
$B^2$ is injective (which is the case). But,
\[D(A)\cap
\ran(B^2)=D(A)\cap D(S)=\{0\}\] as this is already available to us
from Answer 21.2.17 in \cite{Mortad-cex-BOOK} (or \cite{KOS}). So,
$D(AB^2)=\{0\}$. Thus,
\[D(AB^2A)=\{f\in D(A):Af\in D(AB^2)=\{0\}\}=\ker A=\{0\},\]
i.e., $D(AB^2A)$ is not dense in $L^2(\R)$. It only remains to check
that $D(AB)$ is indeed dense in $L^2(\R)$. Observe that by
uniqueness of the positive square root that
\[Bf(x)=e^{\frac{-x^2}{4}}f(x),\]
with $D(B)=L^2(\R)$. Now
\[D(AB)=\{f\in L^2(\R):\widehat{e^{\frac{-x^2}{4}}f}\in L^2(\R), e^{\frac{x^2}{2}}\widehat{e^{\frac{-x^2}{4}}f}\in L^2(\R)\},\]
which is dense by, say, an idea borrowed from the proof of Corollary
9 in \cite{KOS}.

\end{exa}

\begin{rema}
It seems that there is no magic result that guarantees the density
of $D(TT^*)$ or $D(T^*T)$, when $T$ is symmetric, except perhaps to
assume that $T^2$ is densely defined.
\end{rema}

\end{enumerate}

\section*{Acknowledgement}

The author wishes to thank Professor Zsigmond Tarcsay for Example
\ref{Tarcsay'sssssssssssssssssss EXA}, which was based on a
construction by Z. Sebestyén and J. Stochel (see
\cite{Sebestyen-Stochel-Restrictions-positive}).


\begin{thebibliography}{1}


\bibitem{Ando-Paranormal-tensor-product-sums et al}
T. Ando. Operators with a norm condition, \textit{Acta Sci. Math.
(Szeged)}, \textbf{33} (1972) 169-178.

\bibitem{azizov-denisov-philipp products self-adjoint MATH NACH}
T. Ya. Azizov, M. Denisov, F. Philipp. Spectral functions of
products of selfadjoint operators, \textit{Math. Nachr.},
\textbf{285/14-15} (2012) 1711-1728.

\bibitem{Azizov-Dijksma-closedness-prod-ADJ}
T. Ya. Azizov, A. Dijksma. Closedness and adjoints of products of
operators, and compressions, \textit{Integral Equations Operator
Theory}, \textbf{74/2} (2012) 259-269.

\bibitem{Barraa-Boumazghour}
M. Barraa, M. Boumazghour. Numerical range submultiplicity,
\textit{Linear Multilinear Algebra},  \textbf{63/11} (2015),
2311-2317.

\bibitem{Bernau JAusMS-1968-square root}
S. J. Bernau. The square root of a positive self-adjoint operator,
\textit{J. Austral. Math. Soc.}, \textbf{8} (1968) 17-36.


\bibitem{Boucif-Dehimi-Mortad}
I. Boucif, S. Dehimi and M. H. Mortad. On the absolute value of
unbounded operators, \textit{J. Operator Theory}, \textbf{82/2}
(2019) 285-306.


\bibitem{CG}
J. A. W. van Casteren,  S. Goldberg. The conjugate of the product of
operators, \textit{Studia Math.}, {\bf 38} (1970) 125-130.

\bibitem{Cho-Curto-Huruya- sigma(Ab)=sigma(BA) A normal}
M. Ch\={o}, R. E. Curto, T. Huruya. $n$-tuples of operators
satisfying $\sigma_T(AB)=\sigma_T(BA)$. Special issue dedicated to
Professor T. Ando., \textit{Linear Algebra Appl.}, \textbf{341}
(2002) 291-298.

\bibitem{Daniluk-paranormals-non-closable}
A. Daniluk. On the closability of paranormal operators, \textit{J.
Math. Anal. Appl.}, \textbf{376/1} (2011) 342-348.

\bibitem{Dautray-Lions-VOL2}
R. Dautray, J.L. Lions. Mathematical analysis and numerical methods
for science and technology. Vol. \textbf{2}. Functional and
variational methods. With the collaboration of Michel Artola, Marc
Authier, Philippe Bénilan, Michel Cessenat, Jean Michel Combes,
Hélène Lanchon, Bertrand Mercier, Claude Wild and Claude Zuily.
Translated from the French by Ian N. Sneddon.
\textit{Springer-Verlag, Berlin}, 1988.

\bibitem{Dehimi-Mortad-BKMS}
S. Dehimi, M. H. Mortad. Bounded and unbounded operators similar to
their adjoints, \textit{Bull. Korean Math. Soc.}, \textbf{54/1}
(2017) 215-223.

\bibitem{Dehimi-Mortad-INVERT}
S. Dehimi, M. H. Mortad. Right (or left) invertibility of bounded
and unbounded operators and applications to the spectrum of
products, \textit{Complex Anal. Oper. Theory}, \textbf{12/3} (2018)
589-597.


\bibitem{Dehimi-Mortad-CHERNOFF}
S. Dehimi, M. H. Mortad. Chernoff-like counterexamples related to
unbounded operators, \textit{Kyushu J. Math.}, \textbf{74/1} (2020)
105-108.

\bibitem{Dehimi-Mortad-squares-polynomials}
S. Dehimi, M. H. Mortad. Unbounded operators having self-adjoint,
subnormal or hyponormal powers, \textit{Math. Nachr.}, (to appear).
DOI: 10.1002/mana.202100390

\bibitem{Dehimi-Mortad-Bachir-Comm-Closed-Symm-OPER}
S. Dehimi, M. H. Mortad, A. Bachir. On the commutativity of closed
symmetric operators (submitted). arXiv:2203.07266

\bibitem{Dehimi-Mortad-Tarcsay-1}
S. Dehimi, M. H. Mortad, Z. Tarcsay. On the operator equations
$A^n=A^*A$, \textit{Linear Multilinear Algebra}, \textbf{69/9}
(2021) 1771-1778.

\bibitem{Deift}
P. A. Deift. Applications of a commutation formula, \textit{Duke
Math. J.}, \textbf{45/2} (1978) 267-310.

\bibitem{DevNussbaum}
A. Devinatz, A. E. Nussbaum. On the permutability of normal
operators, \textit{Ann. of Math. (2)}, {\bf 65} (1957) 144-152.

\bibitem{DevNussbaum-von-Neumann}
A. Devinatz, A. E. Nussbaum, J. von Neumann. On the permutability of
self-adjoint operators, \textit{Ann. of Math. (2)}, {\bf 62} (1955)
199-203.


\bibitem{Dieudonne-ANALysis-II}
J. Dieudonné. \textit{Treatise on analysis. Vol. II}. Enlarged
and corrected printing. Translated by I. G. Macdonald. With a loose
erratum. Pure and Applied Mathematics, 10-II. Academic Press
[Harcourt Brace Jovanovich, Publishers], New York-London, 1976.

\bibitem{Frid-Mortad-Dehimi-nilpotence}
N. Frid, M. H. Mortad, S. Dehimi. When nilpotence implies the
zeroness of linear operators, \textit{Khayyam J. Math.}, (to
appear).


\bibitem{Gesztesy et al. ANNALI. MATH. AA*}
F. Gesztesy, J. A. Goldstein, H. Holden, G. Teschl. Abstract wave
equations and associated Dirac-type operators, \textit{Ann. Mat.
Pura Appl. (4)}, \textbf{191/4} (2012) 631-676.

\bibitem{Gesztesy-Schmudgen-AA*}
F. Gesztesy, K. Schm\"{u}dgen. On a theorem of Z. Sebestyén and Zs.
Tarcsay, \textit{Acta Sci. Math. (Szeged)}, \textbf{85/1-2} (2019)
291-293.

\bibitem{Gindler spec map THM}
H. A. Gindler. Classroom Notes: A Spectral Mapping Theorem for
Polynomials, \textit{Amer. Math. Monthly}, \textbf{72/5} (1965)
528-530.

\bibitem{Gonzalez ker p(T)}
M. Gonz\'{a}lez. Null spaces and ranges of polynomials of operators,
\textit{Publ. Mat}, \textbf{32/2} (1988) 167-170.

\bibitem{Gustafson-BAMS-PAP self-adjoint}
K. Gustafson. On projections of self-adjoint operators and operator
product adjoints, \textit{Bull. Amer. Math. Soc.}, \textbf{75}
(1969) 739-741.

\bibitem{Gustafson-Mortad-I}
K. Gustafson, M. H. Mortad. Unbounded products of operators and
connections to Dirac-type operators,  \textit{Bull. Sci. Math.},
\textbf{138/5} (2014), 626-642.

\bibitem{Gustafson-Mortad-II}
K. Gustafson, M. H. Mortad. Conditions implying commutativity of
unbounded self-adjoint operators and related topics, \textit{J.
Operator Theory,} \textbf{76/1}, (2016) 159-169.

\bibitem{Hardt-Konstantinov-Spectrum-product}
V. Hardt, A. Konstantinov, R. Mennicken. On the spectrum of the
product of closed operators, \textit{Math. Nachr.}, \textbf{215},
(2000) 91-102.

\bibitem{Hardt-Mennicken-OP-Th-ADv-APP}
V. Hardt, R. Mennicken. On the spectrum of unbounded off-diagonal
$2\times2$ operator matrices in Banach spaces. Recent advances in
operator theory (Groningen, 1998), 243-266, \textit{Oper. Theory
Adv. Appl.}, \textbf{124}, Birkh\"{a}user, Basel, 2001.

\bibitem{HK}
P. Hess, T. Kato. Perturbation of Closed Operators and Their
Adjoints, \textit{Comment. Math. Helv.}, {\bf 45} (1970) 524-529.

\bibitem{Hladnik-Omladic-spectrum-product-PAMS-1988}
M. Hladnik, M. Omladi\v{c}. Spectrum of the product of operators,
\textit{Proc. Amer. Math. Soc.}, {\bf 102/2}, (1988) 300-302.

\bibitem{Jablonski et al 2014}
Z. J. Jab{\l}o\'{n}ski, Il B. Jung, J. Stochel. Unbounded
quasinormal operators revisited, \textit{Integral Equations Operator
Theory,} \textbf{79/1} (2014) 135-149.

\bibitem{Jung-Mortad-Stochel}
Il B. Jung, M. H. Mortad, J. Stochel. On normal products of
selfadjoint operators, \textit{Kyungpook Math. J.}, \textbf{57}
(2017) 457-471.

\bibitem{Kat}
T. Kato. Perturbation theory for linear operators,{\textit{
Springer}}, 1980 (2nd edition).

\bibitem{Kaufman closed oeprators 1983}
W. E. Kaufman. Closed operators and pure contractions in Hilbert
space, \textit{Proc. Amer. Math. Soc.}, \textbf{87/1} (1983)  83-87.

\bibitem{KOS}
H. Kosaki. On intersections of domains of unbounded positive
operators, \textit{Kyushu J. Math.}, {\bf 60/1} (2006) 3-25.

\bibitem{Laberteux-A*A=A2}
K. R. Laberteux. Problem 10377, \textit{Amer. Math. Monthly},
\textbf{101} (1994) 362.

\bibitem{Laursen-Neumann ker p(T)}
K. B. Laursen, M. M. Neumann. An introduction to local spectral
theory. London Mathematical Society Monographs. New Series,
\textbf{20}. \textit{The Clarendon Press, Oxford University Press,
New York}, 2000.

\bibitem{McCullough-Rodman-A*A=A2}
S. A. McCullough, L. Rodman. Hereditary classes of operators and
matrices, \textit{Amer. Math. Monthly,} \textbf{104/5} (1997)
415-430.

\bibitem{McIntosh H infity calculus D(An) dense!!}
A. McIntosh. Operators which have an $H_\infty$ functional calculus.
\textit{Miniconference on operator theory and partial differential
equations} (North Ryde, 1986), 210-231, Proc. Centre Math. Anal.
Austral. Nat. Univ., 14, Austral. Nat. Univ., Canberra, 1986.

\bibitem{Meziane-Mortad-I}
M. Meziane, M. H. Mortad. Maximality of linear operators,
\textit{Rend. Circ. Mat. Palermo, Ser II.}, \textbf{68/3} (2019)
441-451.

\bibitem{MHM1}
M. H. Mortad. An application of the Putnam-Fuglede theorem to normal
products of self-adjoint operators, \textit{Proc. Amer. Math. Soc.},
{\bf 131/10}, (2003) 3135-3141.

\bibitem{MHM7}
M. H. Mortad. On some product of two unbounded self-adjoint
operators, \textit{Integral Equations Operator Theory}, {\bf 64/3}
(2009) 399-408.

\bibitem{Mortad-CMB-2011}
M. H. Mortad. On the adjoint and the closure of the sum of two
unbounded operators, \textit{Canad. Math. Bull.}, {\bf 54/3} (2011)
498-505.

\bibitem{Mortad-OaM-2014}
M. H. Mortad. Commutativity of unbounded normal and self-adjoint
operators and applications, \textit{Oper. Matrices}, \textbf{8/2}
(2014) 563-571.

\bibitem{Mortad-Oper-TH-BOOK-WSPC}
M. H. Mortad. \textit{An operator theory problem book}, World
Scientific Publishing Co., (2018).


\bibitem{Mortad-TRIVIALITY POWERS DOMAINS}
M. H. Mortad. On the triviality of domains of powers and adjoints of
closed operators, \textit{Acta Sci. Math. (Szeged)}, \textbf{85}
(2019) 651-658.


\bibitem{Mortad-cex-BOOK}
M. H. Mortad. \textit{Counterexamples in operator theory},
Birkh\"{a}user/Springer, Cham (2022).

\bibitem{Mortad-paranormal-paper-three CEX CEXEX}
M. H. Mortad. Counterexamples related to unbounded paranormal
operators, \textit{Examples and Counterexamples} (to appear).
https://doi.org/10.1016/j.exco.2021.100017

\bibitem{Nelson-Analytic-vectors}
E. Nelson. Analytic vectors, \textit{Ann. of Math. (2)}, \textbf{70}
(1959) 572-615.

\bibitem{Ota-q deformed quasinormal JOT 2002}
S. \^{O}ta. Some classes of $q$-deformed operators, \textit{J.
Operator Theory}, \textbf{48/1} (2002) 151-186.

\bibitem{Ota-Schmudgen-Matrix-UNBOUNDED}
S. \^{O}ta, K. Schm\"{u}dgen. Some selfadjoint $2\times2$ operator
matrices associated with closed operators, \textit{Integral
Equations Operator Theory}, \textbf{45/4} (2003) 475-484.


\bibitem{Philipp-Ran-Wojtylak}
F. Philipp, A. C. M. Ran, M. Wojtylak. Local definitizability of
$T^{[*]}T$ and $TT^{[*]}$, \textit{Integral Equations Operator
Theory}, \textbf{71/4} (2011) 491-508.

\bibitem{Pietrzycki-Stochel-follow-up-2020-Conjecture curto el al.}
P. Pietrzycki, J. Stochel. On $n$th roots of bounded and unbounded
quasinormal operators.  arXiv:2103.09961v3


\bibitem{Putnam-sq-rt-normal}
C. R. Putnam. On square roots of normal operators, \textit{Proc.
Amer. Math. Soc.}, \textbf{8} (1957) 768-769.

\bibitem{Roman-Sandovici-Mortad-Dehimi-Tarcsay}
M. Roman, A. Sandovici. Multivalued linear operator equation $A^*A =
\lambda A^n$, \textit{Complex Anal. Oper. Theory}, \textbf{15/1}
(2021) Paper No. 6, 14 pp.

\bibitem{RS1}
M. Reed, B. Simon. Methods of modern mathematical physics, Vol. {\bf
1}: \textit{Functional analysis}, Academic Press. 1972.

\bibitem{RS2}
M. Reed, B. Simon. Methods of modern mathematical physics, Vol. {\bf
2}: \textit{Fourier analysis, self-adjointness}, Academic Press.
1975.

\bibitem{SCHMUDG-1983-An-trivial-domain}
K. Schm\"{u}dgen. On domains of powers of closed symmetric
operators, \textit{J. Operator Theory}, {\bf 9/1} (1983) 53-75.

\bibitem{Schmudgen-Nelson LIKE}
K. Schm\"{u}dgen. On commuting unbounded selfadjoint operators. I,
\textit{Acta Sci. Math. (Szeged)}, \textbf{47/1-2} (1984) 131-146.

\bibitem{Schmudgen-Operator-algebra}
K. Schm\"{u}dgen. \textit{Unbounded operator algebras and
representation theory}. Operator Theory: Advances and Applications,
\textbf{37}. Birkh\"{a}user Verlag, Basel, 1990.

\bibitem{SCHMUDG-book-2012}
K. Schm\"{u}dgen. \textit{Unbounded self-adjoint operators on
Hilbert space}, Springer. GTM {\bf 265}  (2012).


\bibitem{Sebestyen-Stochel-Restrictions-positive}
Z. Sebestyén,  J. Stochel. Restrictions of positive selfadjoint
operators, \textit{Acta Sci. Math. (Szeged)}, \textbf{55/1-2} (1991)
149-154.

\bibitem{Sebestyen-Stochel-JMAA}
Z. Sebestyén, J. Stochel. On suboperators with codimension one
domains, \textit{J. Math. Anal. Appl.}, \textbf{360/2} (2009)
391-397.

\bibitem{sebestyen-tarcsay-TT* has an extension}
Z. Sebestyén, Zs. Tarcsay. $T^*T$ always has a positive selfadjoint
extension, \textit{Acta Math. Hungar.}, \textbf{135/1-2} (2012)
116-129.

\bibitem{SEBES-TARCS-CHARATER-SELF6ADJ-OPER}
Z. Sebestyén, Zs. Tarcsay. Characterizations of selfadjoint
operators, \textit{Studia Sci. Math. Hungar.}, \textbf{50/4} (2013)
423-435.


\bibitem{Sebestyen-Tarcsay-TT* von Neumann T closed}
Z. Sebestyén, Zs. Tarcsay. A reversed von Neumann theorem,
\textit{Acta Sci. Math. (Szeged)}, \textbf{80/3-4} (2014) 659-664.

\bibitem{Sebestyen-Tarcsay-self-adjoint squares}
Z. Sebestyén, Zs. Tarcsay. Operators having selfadjoint squares,
\textit{Ann. Univ. Sci. Budapest. E\"{o}tv\"{o}s Sect. Math.},
\textbf{58} (2015) 105-110.

\bibitem{Sebestyen-Tarcsay-adj sum and product}
Z. Sebestyén, Zs. Tarcsay. Adjoint of sums and products of operators
in Hilbert spaces, \textit{Acta Sci. Math. (Szeged)},
\textbf{82/1-2} (2016) 175-191.

\bibitem{Sebestyen-Tarcsay-LMA-2019}
Z. Sebestyén, Zs. Tarcsay. On the adjoint of Hilbert space
operators, \textit{Linear Multilinear Algebra}, \textbf{67/3} (2019)
625-645.


\bibitem{Stochel-IEOT-2002}
J. Stochel. Lifting strong commutants of unbounded subnormal
operators, \textit{Integral Equations Operator Theory},
\textbf{43/2} (2002) 189-214.

\bibitem{Stochel-Szafraniec-normal extensions II}
J. Stochel, F. H. Szafraniec. On normal extensions of unbounded
operators. II, \textit{Acta Sci. Math. (Szeged)}, \textbf{53/1-2}
(1989) 153-177.

\bibitem{Stochel-Sza-domination-1997 powers of paranormals}
J. Stochel, F. H. Szafraniec. $C^\infty$-vectors and boundedness.
Volume dedicated to the memory of W{\l}odzimierz Mlak, \textit{Ann.
Polon. Math.}, \textbf{66} (1997) 223-238.

\bibitem{Stochel-Sza-domination-2003}
J. Stochel, F. H. Szafraniec. Domination of unbounded operators and
commutativity, \textit{J. Math. Soc. Japan}, \textbf{55/2} (2003)
405-437.

\bibitem{Szafraniec paranormal powers kernels}
F. H. Szafraniec. Kato-Protter type inequalities, bounded vectors
and the exponential function, \textit{Ann. Polon. Math.},
\textbf{51} (1990) 303-312.

\bibitem{Szafraniec Normals subnormals and an open question}
F. H. Szafraniec. Normals, subnormals and an open question,
\textit{Oper. Matrices}, \textbf{4/4} (2010) 485-510.

\bibitem{Taylor-Book-1958-FUNC Ana}
A. E. Taylor. \textit{Introduction to functional analysis}, John
Wiley \& Sons, Inc., New York; Chapman \& Hall, Ltd., London 1958.

\bibitem{Thaller-Dirac-EQuation-Contains Nelson trick}
B. Thaller. \textit{The Dirac equation}. Texts and Monographs in
Physics. Springer-Verlag, Berlin, 1992.

\bibitem{Tian DMT AA*=A2}
Y. Tian. A note on two-sided removal and cancellation properties
associated with Hermitian matrix, \textit{preprint} (2021).

\bibitem{Uchiyama-1993-QUASINORMAL}
M. Uchiyama. Operators which have commutative polar decompositions.
Contributions to operator theory and its applications, 197-208,
\textit{Oper. Theory Adv. Appl.}, \textbf{62}, Birkh\"{a}user,
Basel, 1993.

\bibitem{Wang-Zhang}
B-Y. Wang, F. Zhang. Words and normality of matrices, \textit{Linear
and Multilinear Algebra}, \textbf{40/2} (1995) 111-118.

\bibitem{Weidmann}
J. Weidmann. \textit{Linear Operators in Hilbert Spaces}, Springer,
1980.

\bibitem{Yamazaki-Yanagida-PARANORMAL}
T. Yamazaki, M. Yanagida. Relations between two operator
inequalities and their applications to paranormal operators,
\textit{Acta Sci. Math. (Szeged)}, \textbf{69/1-2} (2003) 377-389.

\end{thebibliography}
\end{document}